\documentclass[10pt]{amsart}
\usepackage{amssymb}
\usepackage{amscd}
\usepackage[all]{xy}

\numberwithin{equation}{section}

\def\today{\number\day\space\ifcase\month\or   January\or February\or
March\or April\or May\or June\or   July\or August\or September\or
October\or November\or December\fi\   \number\year}

\newtheorem{Thm}{Theorem}[section]
\newtheorem{Cor}[Thm]{Corollary}
\newtheorem{Lemma}[Thm]{Lemma}

\newtheorem{Prop}[Thm]{Proposition}

\theoremstyle{definition}
\newtheorem{Def}[Thm]{Definition}
\newtheorem{Notation}[Thm]{Notation}

\newtheorem{qst}[Thm]{Question}
\newtheorem{rmk}[Thm]{Remark}

\newtheorem{Construction}[Thm]{Construction}

\newcounter{CnsEnumi}

\newcommand{\etalchar}[1]{$^{#1}$}

\newcommand{\beq}{\begin{equation}}
\newcommand{\eeq}{\end{equation}}
\newcommand{\beqr}{\begin{eqnarray*}}
\newcommand{\eeqr}{\end{eqnarray*}}
\newcommand{\bal}{\begin{align*}}
\newcommand{\eal}{\end{align*}}
\newcommand{\bei}{\begin{itemize}}
\newcommand{\eei}{\end{itemize}}
\newcommand{\limi}[1]{\lim_{{#1} \to \infty}}

\newcommand{\af}{\alpha}
\newcommand{\bt}{\beta}
\newcommand{\gm}{\gamma}
\newcommand{\dt}{\delta}
\newcommand{\ep}{\varepsilon}
\newcommand{\zt}{\zeta}
\newcommand{\et}{\eta}

\newcommand{\io}{\iota}
\newcommand{\te}{\theta}
\newcommand{\ld}{\lambda}
\newcommand{\sm}{\sigma}
\newcommand{\kp}{\kappa}
\newcommand{\ph}{\varphi}
\newcommand{\ps}{\psi}
\newcommand{\rh}{\rho}
\newcommand{\om}{\omega}
\newcommand{\ta}{\tau}

\newcommand{\Gm}{\Gamma}
\newcommand{\Dt}{\Delta}

\newcommand{\Q}{{\mathbb{Q}}}
\newcommand{\Z}{{\mathbb{Z}}}
\newcommand{\R}{{\mathbb{R}}}
\newcommand{\C}{{\mathbb{C}}}
\newcommand{\N}{{\mathbb{Z}}_{> 0}}
\newcommand{\Nz}{{\mathbb{Z}}_{\geq 0}}

\newcommand{\id}{{\operatorname{id}}}
\newcommand{\ev}{{\operatorname{ev}}}

\newcommand{\diag}{{\operatorname{diag}}}

\newcommand{\rank}{{\operatorname{rank}}}

\newcommand{\Aut}{{\operatorname{Aut}}}

\newcommand{\Aff}{{\operatorname{Aff}}}
\newcommand{\Ell}{{\operatorname{Ell}}}

\newcommand{\rc}{\mathrm{rc}}
\newcommand{\tr}{\mathrm{tr}}

\newcommand{\eps}{\varepsilon}
\numberwithin{equation}{section}

\newcommand{\tgamma}{\Gamma}

\newcommand{\cone}{\mathrm{cone}}

\newcommand{\T}{{\operatorname{T}}}

\newcommand{\dirlim}{\varinjlim}
\newcommand{\invlim}{\varprojlim}

\newcommand{\andeqn}{\qquad {\mbox{and}} \qquad}

\newcommand{\ifo}{if and only if}

\newcommand{\ca}{$C^*$-algebra}

\newcommand{\hm}{homomorphism}

\newcommand{\tst}{tracial state}

\newcommand{\pj}{projection}

\newcommand{\ct}{continuous}
\newcommand{\cfn}{continuous function}

\newcommand{\mh}{minimal homeomorphism}

\renewcommand{\S}{\subset}

\newcommand{\I}{\infty}

\title{A simple nuclear C*-algebra with an internal asymmetry}

\author{Ilan Hirshberg and N.~Christopher Phillips}
\address{Department of Mathematics, Ben Gurion University of the Negev,
P.O.B. 653, Be'er Sheva 84105, Israel}
\address{Department of Mathematics, University  of Oregon,
Eugene OR 97403-1222, USA.}

\date{31~July 2020}

\thanks{This material is based upon work supported by the
US National Science Foundation under
Grant DMS-1501144, by the US-Israel Binational Science Foundation
 and by the Israel Science Foundation grant no.~476/16.}

\begin{document}

\begin{abstract}
We construct an example of a simple
approximately homogeneous $C^*$-algebra
such that its Elliott invariant admits an automorphism
which is not induced by an automorphism of the algebra.
\end{abstract}

\maketitle

Classification theory for simple nuclear $C^*$-algebras
reached a milestone recently.
The results of \cite{elliott-gong-lin-niu}
and \cite{tikuisis-white-winter},
building on decades of work by many authors,
show that simple nuclear unital $C^*$-algebras
satisfying the Universal Coefficient Theorem
are classified via the Elliott invariant,
$\Ell (\cdot)$, which consists of the ordered $K_0$-group
along with the class of the identity, the $K_1$-group,
the trace simplex,
and the pairing
between the trace simplex and the $K_0$-group.
Earlier counterexamples
due to Toms and R{\o}rdam
(\cite{toms-counterexample, rordam-counterexample}),
related to ideas of Villadsen (\cite{Villadsen-perforation}),
show that one cannot expect to be able to extend
this classification theorem beyond the case of finite nuclear dimension,
at least without either extending the invariant
or restricting to another class of $C^*$-algebras.
An important facet of the classification theorems
is a form of rigidity.
Starting with two $C^*$-algebras $A$ and $B$
and an isomorphism $\Phi \colon \Ell (A) \to \Ell (B)$,
one not only shows that $A$ and $B$ are isomorphic,
but rather that there exists an isomorphism from $A$ to $B$
which induces the given isomorphism $\Phi$
on the level of the Elliott invariant,
and furthermore that the isomorphism on the algebra level
is unique up to approximate unitary equivalence.

The goal of this paper is to illustrate how this existence property
may fail in the infinite nuclear dimension setting,
even when restricting to a class consisting of a single $C^*$-algebra.
Namely, we construct an example
of a simple unital nuclear separable AH algebra~$C$,
along with an automorphism of $\Ell (C)$,
which is not induced by any automorphism of $C$.
This can be viewed as a companion of sorts to
\cite[Theorem 1.2]{toms-counterexample},
where it was shown that when such automorphisms exist,
they need not be unique in the sense described.
The mechanism of the example is that if there were
such an automorphism~$\ph$,
there would be \pj{s} $p, q \in \C$
such that $\ph (p) = q$
but such that the corners $p C p$ and $q C q$ have different
radii of comparison (\cite{Toms-flat-dim-growth};
the definition is recalled
at the beginning of Section~\ref{Sec_Bounds}).
This further shows that
simple unital AH algebras can be quite inhomogeneous.
In particular,
extending the Elliott invariant by adding something as simple
as the radius of comparison
will not help for the classification of AH~algebras
which are not Jiang-Su stable.

We now give an overview of our construction.
We start with the counterexample
from \cite[Theorem 1.1]{toms-counterexample}.
We consider two direct systems, described diagrammatically as follows:
\begin{equation}\label{Eq_9Y01_PrelimDiagram}
\xymatrix{
C (X_0) \ar@<-0.9ex>@{.>}[r]  \ar[r] \ar@<0.5ex>[r] \ar@<1ex>[r] &
 C (X_1)\otimes M_{r (1)}\ar@<-.9ex>@{.>}[r]
   \ar[r] \ar@<.5ex>[r] \ar@<1ex>[r] &
   C (X_2)\otimes M_{r (2)} \ar@<-.9ex>@{.>}[r]
     \ar[r] \ar@<.5ex>[r] \ar@<1ex>[r] & \cdots \\
C ([0, 1]) \ar@<.9ex>@{.>}[r]
  \ar[r] \ar@<-.5ex>[r] \ar@<-1ex>[r] &
   C ([0, 1])\otimes M_{r (1)}\ar@<0.9ex>@{.>}[r]
  \ar[r] \ar@<-0.5ex>[r] \ar@<-1ex>[r] &
 C ([0, 1])\otimes M_{r (2)} \ar@<.9ex>@{.>}[r]
  \ar[r] \ar@<-.5ex>[r] \ar@<-1ex>[r] & \cdots }
\end{equation}
The ordinary arrows indicate a large (and rapidly increasing)
number of embeddings which are carefully chosen,
and the dotted arrows indicate a small number of point evaluation maps,
thrown in so as to ensure that the resulting direct limit is simple.
The spaces in the upper diagram are contractible CW complexes
whose dimension increases rapidly
compared to the sizes of the matrix algebras.
(Toms uses cubes; in our construction
we found it easier to use cones over products of spheres,
but the underlying idea is similar.)
The direct system is constructed so as to have
positive radius of comparison.
We use~\cite{thomsen} to choose
the lower diagram so as to mimic the upper diagram,
and produce the same Elliott invariant.
As the resulting algebra on the bottom is AI,
it has strict comparison,
and therefore is not isomorphic to the one on the top.
(In \cite{toms-counterexample} it isn't important
for the two diagrams to match up nicely
in terms of the ranks of the matrices involved.
However, we will show that it can be done,
as it is important for us.)

Our construction involves moving the point evaluations across,
so as to merge the two systems,
getting:
\begin{equation}\label{Eq_9Y01_Diagram}
\xymatrix{
C (X_0) \ar@{.>}[dr]  \ar[r] \ar@<.5ex>[r] \ar@<1ex>[r]
 & C (X_1)\otimes M_{r (1)}\ar@{.>}[dr]
   \ar[r] \ar@<.5ex>[r] \ar@<1ex>[r]
 & C (X_2)\otimes M_{r (2)} \ar@{.>}[dr]
    \ar[r] \ar@<.5ex>[r] \ar@<1ex>[r] & \cdots
     \\
C ([0, 1]) \ar@{.>}[ur]  \ar[r] \ar@<-.5ex>[r]
 \ar@<-1ex>[r] & C ([0, 1]) \otimes M_{r (1)} \ar@{.>}[ur]
   \ar[r] \ar@<-.5ex>[r]
   \ar@<-1ex>[r]& C ([0, 1]) \otimes M_{r (2)} \ar@{.>}[ur]
    \ar[r] \ar@<-.5ex>[r] \ar@<-1ex>[r] & \cdots .}
\end{equation}
With care, one can arrange for the flip
between the two levels of the diagram
to make sense as an automorphism of the Elliott invariant.
The resulting $C^*$-algebra
has positive radius of comparison
and behaves roughly as badly as Toms' example.
Nevertheless, we can distinguish a part of it
which roughly corresponds to the rapid dimension growth diagram
on the top
from a part which roughly corresponds to the AI part on the bottom.
Namely, if at the first level $C (X_0) \oplus C ([0, 1])$
we denote by $q$ the function which is $1$ on $X_0$ and $0$ on $[0, 1]$,
and we denote $q^{\perp} = 1 - q$,
then the $K_0$-classes of $q$ and $q^{\perp}$
will be switched by the automorphism of the Elliott invariant
we construct.
However, we can tell apart the corners $q C q$
and $q^{\perp} C q^{\perp}$
by considering their radii of comparison.

Section~\ref{Sec_Bounds} develops the choices needed to
get different radii of comparison in different corners
of the algebra we construct.
Section~\ref{Sec_Traces} contains the work needed to assemble
the ingredients of the construction
into a simple \ca{}
whose Elliott invariant admits an appropriate automorphism.
The main theorem is in Section~\ref{Sec_Main}.

The second author is grateful to M.~Ali Asadi-Vasfi
for a careful reading of Section~\ref{Sec_Bounds},
and in particular finding a number of misprints.

\section{Upper and lower bounds on the radius of
 comparison}\label{Sec_Bounds}

We recall the required standard definitions and notation
related to the Cuntz semigroup.
See Section~2 of~\cite{rordam-tensored-by-uhf-ii} for details.
For a unital \ca~$A$,
we denote its tracial state space by $\T (A)$.
We take $M_{\I} (A) = \bigcup_{n = 1}^{\I} M_n (A)$,
using the usual embeddings
$M_n (A) \hookrightarrow M_{n + 1} (A)$.
For $\ta \in \T (A)$,
we define $d_{\ta} \colon M_{\I} (A)_{+} \to [0, \infty)$
by $d_{\ta} (a) = \lim_{n \to \infty} \ta (a^{1/n})$.
If $a, b \in M_{\I} (A)_{+}$,
then $a \precsim b$
($a$ is Cuntz subequivalent to~$b$)
if there is a sequence $(v_n)_{n = 1}^{\infty}$ in $M_{\I} (A)$
such that
$\limi{n} v_n b v_n^* = a$.

Following \cite[Definition 6.1]{Toms-flat-dim-growth},
for $\rh \in [0, \I)$,
we say that $A$ has $\rh$-comparison if whenever
$a, b \in M_{\infty} (A)_{+}$ satisfy
$d_{\ta} (a) + \rh < d_{\ta} (b)$
for all $\ta \in \T (A)$,
then $a \precsim b$.
The radius of comparison of~$A$,
denoted ${\operatorname{rc}} (A)$, is
\[
{\operatorname{rc}} (A)
 = \inf \big( \big\{ \rh \in [0, \I) \mid
    {\mbox{$A$ has $\rh$-comparison}} \big\} \big) \, .
\]
We take ${\operatorname{rc}} (A) = \infty$ if there is no $\rh$
such that $A$ has $\rh$-comparison.
Since AH~algebras are nuclear,
all quasitraces on them are traces by
\cite[Theorem 5.11]{haagerup-quasitraceas}.
Thus, we ignore quasitraces.
Also,
by Proposition 6.12 of~\cite{Phl40},
the radius of comparison remains unchanged
if we replace $M_{\I} (A)$ by $K \otimes A$ throughout.
Thus, we may work only in $M_{\I} (A)$.

Our construction uses a specific setup,
with a number of parameters of various kinds
which must be chosen to satisfy specific conditions.
Construction \ref{Cn_6918_General} lists for reference
many of the objects used in it,
and some of the conditions they must satisfy.
It abstracts the diagram~(\ref{Eq_9Y01_Diagram}).
Construction~\ref{Cn_6918_General_Part_1a}
specifies the choices of spaces and maps needed for the
results on Cuntz comparison,
and Construction~\ref{Cn_6921_GeneralPart2}
together with the additional maps
in parts (\ref{Cn_6918_General_2nd}),
(\ref{Cn_9Y01_Gen_2ndDiag}),
and~(\ref{Cn_6918_General_Agree})
of Construction~\ref{Cn_6918_General},
is used to arrange the existence of a suitable automorphism
of the tracial state space of the algebra we construct.
Because of the necessity of passing to a subsystem at one
stage in this process,
we must start the proof of the main theorem with a version
of just the top row in the diagram~(\ref{Eq_9Y01_PrelimDiagram});
this is Construction~\ref{Cn_9104_Half}.
Many of the lemmas use only a few of the
objects and their properties,
so that the reader can refer back to just the relevant
parts of the constructions.
In particular,
many details are used only in this section
or only in Section~\ref{Sec_Traces}.
Some of the details are used for just one lemma each.

\begin{Construction}\label{Cn_6918_General}
For much of this paper,
we will consider algebras constructed
in the following way
and using the following notation:
\begin{enumerate}
\item\label{Cn_6918_General_dn}
$(d (n) )_{n = 0, 1, 2, \ldots}$
and $(k (n) )_{n = 0, 1, 2, \ldots}$
are sequences in~$\Nz$,
with $d (0) = 1$ and $k (0) = 0$.
Moreover,
for $n \in \Nz$,
\[
l (n) = d (n) + k (n) \, ,
\qquad
r (n) = \prod_{j = 0}^n l (j) \, ,
\andeqn
s (n) = \prod_{j = 0}^n d (j) \, .
\]
Further define $t (n)$ inductively as follows.
Set $t (0) = 0$, and
\[
t (n + 1) = d (n + 1) t (n) + k (n + 1) [r (n) - t (n)] \, .
\]
(See Lemma~\ref{L_7816_Rankq} for the significance of $t (n)$.)
\item\label{Cn_6918_General_dnkn}
We will assume that $k (n) < d (n)$
for all $n \in \Nz$.
\item\label{Cn_6918_General_Size}
We define
\[
\kp = \inf_{n \in \N} \frac{s (n)}{r (n)} \, .
\]
For estimates involving the radius of comparison,
we will assume $\kp > \frac{1}{2}$.
\item\label{Cn_6918_General_rrp}
The numbers $\om, \om' \in (0, \I]$ are defined by
\[
\om = \frac{k (1)}{k (1) + d (1)}
\andeqn
\om' = \sum_{n = 2}^{\infty} \frac{k (n)}{k (n) + d (n)} \, .
\]
We will require
$\om' < \om < \frac{1}{2}$.
In particular,
\[
\sum_{n = 1}^{\infty} \frac{k (n)}{k (n) + d (n)} < \infty \, .
\]
\item\label{Cn_6918_General_kpom}
We will also eventually require that
$\kp$ as in (\ref{Cn_6918_General_Size})
and $\om$ as in (\ref{Cn_6918_General_rrp})
are related by $2 \kappa - 1 > 2 \om$.
\item\label{Cn_6918_General_Spaces}
$(X_n)_{n = 0, 1, 2, \ldots}$ and $(Y_n)_{n = 0, 1, 2, \ldots}$
are sequences of compact metric spaces.
(They will be further specified
in Construction \ref{Cn_6918_General_Part_1a}.)
\item\label{Cn_6918_General_Algs}
For $n \in \Nz$,
the algebra $C_n$
is
\[
C_n = M_{r (n)} \otimes \big( C (X_n) \oplus C (Y_n) \big) \, .
\]
We further make the identifications:
\begin{align*}
& C (X_{n + 1}, \, M_{r (n + 1)})
= M_{l (n + 1)} \otimes C (X_{n + 1}, \, M_{r (n)}),
\\
& C (Y_{n + 1}, \, M_{r (n + 1)})
= M_{l (n + 1)} \otimes C (Y_{n + 1}, \, M_{r (n)}),
\\
& 	C (X_n) \oplus C (Y_n) = C (X_n \amalg Y_n),
\\
 & 	C (X_n, \, M_{r (n)}) \oplus C (Y_n, \, M_{r (n)})
 = C (X_n \amalg Y_n, M_{r (n)}) \, .
\end{align*}
\item\label{Cn_6918_General_Maps}
For $n \in \N$,
we are given a unital \hm
\[
\gm_n \colon
C (X_n) \oplus C (Y_n)
\to M_{l (n + 1)} \big( C (X_{n + 1}) \oplus C (Y_{n + 1}) \big) \, ,
\]
and the \hm{}
\[
\Gm_{n + 1, \, n} \colon C_n \to C_{n + 1}
\]
is given by
$\Gm_{n + 1, \, n} = \id_{M_{r (n)}} \otimes \gm_n$.
Moreover,
for $m, n \in \Nz$ with $m \leq n$,
\[
\Gm_{n, m}
= \Gm_{n, n - 1} \circ \Gm_{n - 1, \, n - 2} \circ \cdots
\circ \Gm_{m + 1, m}
\colon C_m \to C_n \, .
\]
In particular, $\Gm_{n, n} = \id_{C_n}$.
\item\label{Cn_6918_General_Diag}
We require that the maps
\[
\gm_n \colon C (X_n \amalg Y_n)
\to M_{l (n + 1)} \bigl( C (X_{n + 1} \amalg Y_{n + 1}) \bigr)
\]
in~(\ref{Cn_6918_General_Maps})
be diagonal,
that is,
that there exist
\cfn{s}
\[
S_{n, 1}, \, S_{n, 2}, \, \ldots, \, S_{n, \, l (n + 1)}
\colon X_{n + 1} \amalg Y_{n + 1} \to X_n \amalg Y_n
\]
such that for all $f \in C (X_n \amalg Y_n)$,
we have
\[
\gm_n (f)
= \diag \big( f \circ S_{n, 1}, \, f \circ S_{n, 2},
\, \ldots, \, f \circ S_{n, \, l (n + 1)} \big) \, .
\]
(These maps will be specified further
in Construction~\ref{Cn_6918_General_Part_1a}.)
\item\label{Cn_6918_General_Limit}
We set $C = \dirlim_n C_n$,
taken with respect to the maps $\Gm_{n, m}$.
The maps associated with the direct limit
will be called $\Gm_{\infty, m} \colon C_m \to C$
for $m \in \Nz$.
\setcounter{CnsEnumi}{\value{enumi}}
\end{enumerate}

We sometimes use
additional objects and conditions in the construction,
as follows:
\begin{enumerate}
\setcounter{enumi}{\value{CnsEnumi}}
\item\label{Cn_6918_General_2nd}
For $n \in \N$,
we may be given an additional unital \hm
\[
\gm_n^{(0)} \colon
C (X_n) \oplus C (Y_n)
\to M_{l (n + 1)} \big( C (X_{n + 1}) \oplus C (Y_{n + 1}) \big) \, .
\]
Then the maps
$\Gm_{n + 1, \, n}^{(0)} \colon C_n \to C_{n + 1}$,
$\Gm_{n, m}^{(0)} \colon C_m \to C_n$
are defined analogously to~(\ref{Cn_6918_General_Maps}),
the algebra $C^{(0)}$ is given as
$C^{(0)} = \dirlim_n C_n$,
taken with respect to the maps $\Gm_{n, m}^{(0)}$,
and the maps $\Gm_{\infty, m}^{(0)} \colon C_m \to C^{(0)}$
are defined analogously to~(\ref{Cn_6918_General_Limit}).
\item\label{Cn_9Y01_Gen_2ndDiag}
In~(\ref{Cn_6918_General_2nd}),
analogously to~(\ref{Cn_6918_General_Diag}),
we may require that there be
\[
S_{n, 1}^{(0)}, \, S_{n, 2}^{(0)},
   \, \ldots, \, S_{n, \, l (n + 1)}^{(0)}
\colon X_{n + 1} \amalg Y_{n + 1} \to X_n \amalg Y_n
\]
such that for all $f \in C (X_n \amalg Y_n)$,
we have
\[
\gm_n^{(0)} (f)
= \diag \big( f \circ S_{n, 1}^{(0)}, \, f \circ S_{n, 2}^{(0)},
\, \ldots, \, f \circ S_{n, \, l (n + 1)}^{(0)} \big) \, .
\]
(These maps will be specified further
in Construction~\ref{Cn_6918_General_Part_1a}.)
\item\label{Cn_6918_General_Agree}
Assuming diagonal maps as in~(\ref{Cn_6918_General_Diag}),
we may require that they agree
in the coordinates $1, 2, \ldots, d (n + 1)$,
that is,
for $n \in \N$ and
$k = 1, 2, \ldots, d (n + 1)$,
we have $S_{n, k}^{(0)} = S_{n, k}$.
\setcounter{CnsEnumi}{\value{enumi}}
\end{enumerate}
\end{Construction}

\begin{Lemma}\label{L_6922_Noninc}
In Construction \ref{Cn_6918_General}(\ref{Cn_6918_General_dn}),
the sequence $\left ( \frac{s (n)}{r (n)} \right )_{n = 1, 2, \ldots}$
is strictly decreasing.
\end{Lemma}

\begin{proof}
The proof is straightforward.
\end{proof}

\begin{Lemma}\label{L_7814_Inc}
In Construction \ref{Cn_6918_General}(\ref{Cn_6918_General_dn}),
and assuming
Construction \ref{Cn_6918_General}(\ref{Cn_6918_General_dnkn}),
we have
\[
0 = \frac{t (0)}{r (0)}
  < \frac{t (1)}{r (1)}
  < \frac{t (2)}{r (2)}
  < \cdots
  < \frac{1}{2} \, .
\]
\end{Lemma}

\begin{proof}
We have $t (0) = 0$ by definition.
We prove by induction on~$n \in \N$ that
\begin{equation}\label{Eq_7814_IndHyp}
\frac{t (n - 1)}{r (n - 1)}
 < \frac{t (n)}{r (n)}
 < \frac{1}{2}.
\end{equation}
This will finish the proof.
For $n = 1$,
we have
\[
\frac{t (1)}{r (1)}
 = \frac{k (1)}{k (1) + d (1)} \, ,
\]
which is in $\big( 0, \frac{1}{2} \big)$
by Construction \ref{Cn_6918_General}(\ref{Cn_6918_General_dnkn}).
Now assume~(\ref{Eq_7814_IndHyp});
we prove this relation with $n + 1$ in place of~$n$.
We have $r (n) - t (n) > t (n)$, so
\begin{align}
\frac{t (n + 1)}{r (n + 1)}
& = \frac{d (n + 1) t (n) + k (n + 1) [r (n) - t (n)]}{
       [d (n + 1) + k (n + 1)] r (n)}
  \label{Eq_7816_Start}
\\
& > \frac{d (n + 1) t (n) + k (n + 1) t (n)}{
       [d (n + 1) + k (n + 1)] r (n)}
  = \frac{t (n)}{r (n)} \, .
  \notag
\end{align}
Also,
with
\[
\af = \frac{d (n + 1)}{d (n + 1) + k (n + 1)}
\andeqn
\bt = \frac{t (n)}{r (n)} \, ,
\]
starting with the first step in~(\ref{Eq_7816_Start}),
and at the end using $\af > \frac{1}{2}$
(by Construction \ref{Cn_6918_General}(\ref{Cn_6918_General_dnkn}))
and $\bt < \frac{1}{2}$ (by the induction hypothesis),
we have
\[
\frac{t (n + 1)}{r (n + 1)}
 = \af \bt + (1 - \af) (1 - \bt)
 = \frac{1}{2} \big[ 1 - (2 \af - 1) (1 - 2 \bt) \big]
 < \frac{1}{2} \, .
\]
This completes the induction, and the proof.
\end{proof}

\begin{Lemma}\label{L_6922_Rho}
With the notation of
Construction \ref{Cn_6918_General}(\ref{Cn_6918_General_dn})
and Construction \ref{Cn_6918_General}(\ref{Cn_6918_General_rrp}),
and assuming the conditions
in Construction \ref{Cn_6918_General}(\ref{Cn_6918_General_dnkn})
and Construction \ref{Cn_6918_General}(\ref{Cn_6918_General_rrp}),
for all $n \in \N$
we have
\[
\omega
\leq \frac{t (n)}{r (n)}
\leq \omega + \omega'
< 2 \omega \, .
\]
\end{Lemma}

\begin{proof}
The third inequality is immediate from
Construction \ref{Cn_6918_General}(\ref{Cn_6918_General_rrp}).

By
Lemma~\ref{L_7814_Inc},
the sequence
$\left ( \frac{t (n)}{r (n)} \right )_{n = 1, 2, \ldots}$
is strictly increasing.
Also,
\begin{equation}\label{Eq_7806_StSt}
\frac{t (1)}{r (1)}
 = \frac{k (1)}{k (1) + d (1)}
 = \om.
\end{equation}
The first inequality in the statement now follows.

Next, we claim that
\[
\frac{t (n)}{r (n)}
 \leq \sum_{j = 1}^{n} \frac{k (j)}{k (j) + d (j)}
\]
for all $n \in \N$.
The case $n = 1$ is~(\ref{Eq_7806_StSt}).
Assume this inequality is known for~$n$.
Then
\begin{align*}
\frac{t (n + 1)}{r (n + 1)}
& = \left( \frac{d (n + 1)}{k (n + 1) + d (n + 1)} \right)
      \left( \frac{t (n)}{r (n)} \right)
\\
& \hspace*{3em} {\mbox{}}
     + \left( \frac{k (n + 1)}{k (n + 1) + d (n + 1)} \right)
         \left( \frac{r (n) - t (n)}{r (n)} \right)
\\
& \leq \frac{t (n)}{r (n)}
     + \frac{k (n + 1)}{k (n + 1) + d (n + 1)}
  \leq \sum_{j = 1}^{n + 1} \frac{k (j)}{k (j) + d (j)} \, ,
\end{align*}
as desired.

The second inequality in the statement now follows.
\end{proof}

\begin{Notation}\label{N_7730_Cone}
For a topological space $X$, we define
\[
\cone (X) = (X \times [0, 1]) / (X \times \{ 0 \}) \, .
\]
Then $\cone (X)$ is contractible,
and $\cone ( \cdot )$ is a covariant functor:
if $T \colon X \to Y$ is a continuous map,
then it induces a continuous map
$\cone (T) \colon \cone (X) \to \cone (Y)$.
We identify $X$ with the image of $X \times \{ 1 \}$ in $\cone (X)$.
\end{Notation}

\begin{Construction}\label{Cn_6918_General_Part_1a}
We give further details
on the spaces $X_n$ and $Y_n$
in Construction \ref{Cn_6918_General}(\ref{Cn_6918_General_Spaces}).
\begin{enumerate}
\setcounter{enumi}{\value{CnsEnumi}}
\item\label{Cn_6918_General_1a_Xn}
The space $X_n$ is chosen as follows.
First set $Z_0 = S^2$.
With $(d (n) )_{n = 0, 1, 2, \ldots}$
and $(s (n) )_{n = 0, 1, 2, \ldots}$
as in Construction \ref{Cn_6918_General}(\ref{Cn_6918_General_dn}),
define inductively
\[
Z_{n} = Z_{n - 1}^{d (n)} = (S^2)^{s (n)}.
\]
Then set $X_n = \cone (Z_n)$.
(In particular, $X_n$ is contractible,
and $Z_n \S X_n$ as in Notation~\ref{N_7730_Cone}.)
Further, for $n \in \Nz$ and $j = 1, 2, \ldots, d (n + 1)$,
we let $P^{(n)}_j \colon Z_{n + 1} \to Z_{n}$
be the $j$-th coordinate projection,
and we set
$Q^{(n)}_j = \cone \big( P^{(n)}_j \big) \colon X_{n + 1} \to X_{n}$.
\item\label{Cn_6918_General_1a_YnIs01}
$Y_n = [0, 1]$ for all $n \in \N$.
(In particular, $Y_n$ is contractible.)
\item\label{condition-points-dense}
We assume we are given points $x_m \in X_m$
for $m \in \Nz$ such that,
using the notation in~(\ref{Cn_6918_General_1a_Xn}),
for all $n \in \Nz$,
the set
\begin{align*}
& \big\{ \big( Q^{(n)}_{\nu_{1}} \circ Q^{(n + 1)}_{\nu_{2}}
      \circ \cdots \circ Q^{(m - 1)}_{\nu_{m - n}} \big) (x_m) \mid
\\
& \hspace*{1em} {\mbox{}}
       {\mbox{$m = n + 1, \, n + 2, \, \ldots$
       and $\nu_j = 1, 2, \ldots, d (n + j)$
       for $j = 1, 2, \ldots, m - n$}} \big\}
\end{align*}
is dense in $X_n$.
\item\label{Cn_6918_General_1a_yn}
We assume we are given a sequence $(y_k)_{k = 0, 1, 2, \ldots}$
in $[0, 1]$
such that for all $n \in \Nz$, the set
$\{ y_k \mid k \geq n \}$ is dense in $[0, 1]$.
\item\label{Cn_6918_General_1a_Maps}
The maps
\[
\gm_n \colon C (X_n \amalg Y_n)
\to M_{l (n + 1)} \bigl( C (X_{n + 1} \amalg Y_{n + 1}) \bigr)
\]
will be as in
Construction \ref{Cn_6918_General}(\ref{Cn_6918_General_Diag}),
with the maps
$S_{n, j} \colon X_{n + 1} \amalg Y_{n + 1} \to X_n \amalg Y_n$
appearing there defined as follows:
\begin{enumerate}
\item\label{Cn_6918_General_1a_Maps_XLow}
With $Q^{(n)}_j$ as in~(\ref{Cn_6918_General_1a_Xn}),
we set $S_{n, j} (x) = Q^{(n)}_j (x)$
for $x \in X_{n + 1}$ and $j = 1, 2, \ldots, d (n + 1)$.
\item\label{Cn_6918_General_1a_Maps_XHigh}
$S_{n, j} (x) = y_n$
for
\[
x \in X_{n + 1}
\andeqn
j = d (n + 1) + 1, \, d (n + 1) + 2, \, \ldots, \, l (n + 1) \, .
\]
\item\label{Cn_6918_General_1a_Maps_YBottom}
There are \cfn{s}
\[
R_{n, 1}, \, R_{n, 2}, \, \ldots, \, R_{n, \, d (n + 1)}
\colon Y_{n + 1} \to Y_n
\]
(which will be taken from Proposition \ref{P_arbitrary_trace_space}
below)
such that
$S_{n, j} (y) = R_{n, j} (y)$
for $y \in Y_{n + 1}$ and $j = 1, 2, \ldots, d (n + 1)$.
\item\label{Cn_6918_General_1a_Maps_YTop}
$S_{n, j} (y) = x_n$
for
\[
y \in Y_{n + 1}
\andeqn
j = d (n + 1) + 1, \, d (n + 1) + 2, \, \ldots, \, l (n + 1) \, .
\]
\end{enumerate}
\item\label{Cn_9Y01_Gen_1a_Maps2}
The maps
\[
\gm_n^{(0)} \colon C (X_n \amalg Y_n)
\to M_{l (n + 1)} \bigl( C (X_{n + 1} \amalg Y_{n + 1}) \bigr)
\]
will be as in
Construction \ref{Cn_6918_General}(\ref{Cn_9Y01_Gen_2ndDiag}),
with the maps
$S_{n, j}^{(0)} \colon X_{n + 1} \amalg Y_{n + 1} \to X_n \amalg Y_n$
appearing there given by
$S_{n, j}^{(0)} = S_{n, j}$ for $j = 1, 2, \ldots, d (n + 1)$
and to be specified later
for $j = d (n + 1) + 1, \, d (n + 1) + 2, \, \ldots, \, l (n + 1)$.
\setcounter{CnsEnumi}{\value{enumi}}
\end{enumerate}
\end{Construction}

With the choices in Construction
\ref{Cn_6918_General_Part_1a}(\ref{Cn_6918_General_1a_Maps}),
the map
\[
\gamma_{n} \colon C (X_n) \oplus C (Y_n)
  \to C (X_{n + 1}, \, M_{l (n + 1)})
      \oplus C (Y_{n + 1}, \, M_{l (n + 1)})
\]
in Construction \ref{Cn_6918_General}(\ref{Cn_6918_General_Maps}),
as further specified
in Construction \ref{Cn_6918_General}(\ref{Cn_6918_General_Diag}),
is given as follows.
With $\C^{d (n)}$ viewed as embedded in $M_{d (n)}$
as the diagonal matrices,
there is a \hm{}
\[
\delta_{n} \colon C (Y_n) \to C (Y_{n + 1}, \, \C^{d (n + 1)})
  \subset C (Y_{n + 1}, M_{d (n + 1)})
\]
such that
\begin{align}
\gamma_{n} (f, g)
& =
\Bigg( \diag \Big(
   f \circ Q^{(n)}_1, \, f \circ Q^{(n)}_2, \, \ldots, \,
   f \circ Q^{(n)}_{d (n + 1)}, \,
 \underbrace{g (y_n), \, g (y_n), \, \ldots, \,
  g (y_n)}_{\mbox{$k (n + 1)$ times}} \Big),
 \notag
\\
& \hspace*{3em} {\mbox{}}
 \diag \Big( \delta_{n} (g), \,
 \underbrace{f (x_n), \, f (x_n), \, \ldots, \,
   f (x_n)}_{\mbox{$k (n + 1)$ times}} \Big) \Bigg) \, .
 \label{Eq_7830_Star}
\end{align}
For the purposes of this section,
we need no further information on the maps~$\dt_n$,
except that they send constant functions to constant functions.

\begin{Lemma}\label{L_0725_Simplicity}
Assume the notation and choices in
parts (\ref{Cn_6918_General_dn}),
(\ref{Cn_6918_General_Algs}),
(\ref{Cn_6918_General_Maps}), and~(\ref{Cn_6918_General_Limit})
of Construction~\ref{Cn_6918_General},
and in Construction \ref{Cn_6918_General_Part_1a}
(except part~(\ref{Cn_9Y01_Gen_1a_Maps2}))
and the parts of Construction~\ref{Cn_6918_General}
referred to there.
Then the algebra $C$ is simple.
\end{Lemma}

As the proof is a standard argument
using Construction
\ref{Cn_6918_General_Part_1a}(\ref{condition-points-dense}),
we omit it.

\begin{Notation}\label{N_7806_Bott}
Let $p \in C (S^2, M_2)$ denote the Bott projection, and let
$L$ be the tautological line bundle over
$S^2 \cong \C \mathbb{P}^1$.
(Thus, the range of~$p$ is the section space of~$L$.)
Recalling that $X_0 = \cone (S^2)$,
parametrized as in Notation~\ref{N_7730_Cone},
define
$b \in C (X_0, M_2)$ by $b (\ld) = \ld \cdot p$
for $\ld \in [0, 1]$.
Assuming the notation and choices in
parts (\ref{Cn_6918_General_dn}),
(\ref{Cn_6918_General_Spaces}),
(\ref{Cn_6918_General_Algs}),
(\ref{Cn_6918_General_Maps}), and~(\ref{Cn_6918_General_Limit})
of Construction~\ref{Cn_6918_General}
and in Construction \ref{Cn_6918_General_Part_1a},
for $n \in \Nz$
set
$b_n = (\id_{M_2} \otimes \tgamma_{n, 0}) (b, 0)
     \in M_2 (C_n)$.
\end{Notation}

We require the following simple lemma
concerning characteristic classes.
It gives us a way of estimating the radius of comparison
which is similar
to the one used in \cite[Lemma 1]{Villadsen-perforation},
but more suitable for the types of estimates we need here.

\begin{Lemma}\label{Lemma:Chern-class}
The Cartesian product $L^{\times k}$
does not embed in a trivial bundle over $(S^2)^k$
of rank less than $2k$.
\end{Lemma}

\begin{proof}
We refer the reader to \cite[Section 14]{milnor-stasheff}
for an account of Chern classes.
The Chern character $c (L)$ is of the form $1 + \eps$,
where $\eps$ is a generator of $H^2 (S^2, \Z)$,
and the product operation satisfies $\eps^2 = 0$.
Let $P_1, P_2, \ldots, P_k \colon (S^2)^k \to S^2$
be the coordinate projections.
For $j = 1, 2, \ldots, k$, set $\eps_j = P_j^* (\eps)$.
The elements $\eps_1, \eps_2, \ldots, \eps_k \in H^2 ((S^2)^k, \Z)$,
along with $1 \in H^0 ((S^2)^k, \Z)$ (the standard generator)
generate the cohomology ring of $(S^2)^k$,
and satisfy $\eps_j^2 = 0$ for $j = 1, 2, \ldots, k$.
By naturality of the Chern character
(\cite[Lemma 14.2]{milnor-stasheff})
and the product theorem (\cite[(14.7) on page 164]{milnor-stasheff}),
we have $c (L^{\times k}) = \prod_{j = 1}^k (1 + \eps_j)$.
Now, suppose $L^{\times k}$ embeds as a subbundle
of a trivial bundle~$E$.
Let $F$ be the complementary bundle,
so that $L^{\times k} \oplus F = E$.
By the product theorem,
$c (L^{\times k}) c (F) = c (L^{\times k}\oplus F) = c (E) = 1$.
Thus, $c (F) = c (L^{\times k})^{-1} = \prod_{j = 1}^k (1 - \eps_j)$.
Since $c (F)$ has a nonzero term in the top cohomology class
$H^{2 k} ((S^2)^{k})$,
it follows that $\rank (F)$ is at least $k$.
Thus, $\rank (E) = \rank (L^{\times k}) + \rank (F) \geq 2k$,
as required.
\end{proof}

\begin{Lemma}\label{L_7906_RankB}
Adopt the assumptions and notation of Notation~\ref{N_7806_Bott}.
Let $n \in \N$.
Then $b_n |_{Z_n}$ is the orthogonal sum of a
projection $p_n$ whose range
is isomorphic to the section space
of the Cartesian product bundle $L^{\times s (n)}$
and a constant function of rank at most
$r (n) - s (n) - t(n)$.
\end{Lemma}

We don't expect $b_n |_{Z_n}$ to be a projection,
since some of the point evaluations
occurring in the maps of the direct system will be at
points $x \in \cone (Z_m) \setminus Z_m$
for values of $m < n$,
and $b_m (x)$ is not a \pj{} for such~$x$.

We don't need the estimate on the rank of the second part of
the description of $b_n |_{Z_n}$;
it is included to make the construction more explicit.
If there are no evaluations at the ``cone points''
\[
(Z_m \times \{ 0 \}) / (Z_m \times \{ 0 \})
 \in (Z_m \times [0, 1]) / (Z_m \times \{ 0 \})
\]
(following the parametrization in Notation~\ref{N_7730_Cone}),
then this rank
will be exactly $r (n) - s (n) - t(n)$.

\begin{proof}[Proof of Lemma~\ref{L_7906_RankB}]
For $n \in \Nz$ write
$b_n = (c_n, g_n)$
with
\[
c_n \in M_2 (C (X_n, M_{r (n)} ))
\andeqn
g_n \in M_2 (C (Y_n, M_{r (n)} )) \, .
\]
Further,
for $j = 1, 2, \ldots, s (n)$
let $T_j^{(n)} \colon (S^2)^{s (n)} \to S^2$
be the $j$-th coordinate projection.
We claim that
$c_n$ is an orthogonal sum $c_{n, 0} + c_{n, 1}$,
in which $c_{n, 0}$ is the direct sum of
the functions $b \circ \cone \big( T_j^{(n)} \big)$
for $j = 1, 2, \ldots, s (n)$
and $c_{n, 1}$ is a constant function of rank at most
$r (n) - s (n) - t (n)$,
and moreover that $g_n$
is a constant function of rank at most $t (n)$.
The statement of the lemma follows from this claim.

The proof of the claim is by induction on~$n$.
The claim is true for $n = 0$,
by the definition of~$b$
and since $s (0) = 1$,
$t (0) = 0$,
and $r (0) - s (0) - t (0) = 0$.

Now assume that the claim is known for~$n$,
recall that $\Gm_{n + 1, \, n} = \id_{M_{r (n)}} \otimes \gm_n$
(see Construction \ref{Cn_6918_General}(\ref{Cn_6918_General_Maps})),
and examine the summands in the description~(\ref{Eq_7830_Star})
of the map $\gamma_{n}$
(after Construction \ref{Cn_6918_General_Part_1a}).
With this convention,
first take $(f, g)$ in~(\ref{Eq_7830_Star})
to be $( c_{n, 0}, 0 )$.
The first coordinate
$\Gamma_{n + 1, n} ( c_{n, 0}, 0 )_1$
is of the form required for $c_{n + 1, 0}$,
while
$\Gamma_{n + 1, n} ( c_{n, 0}, 0 )_2$
is a constant function of rank $k (n + 1) s (n)$
unless $c_n (x_n) = 0$,
in which case it is zero.
In the same manner, we see that:
\begin{itemize}
\item
$\Gamma_{n + 1, n} ( c_{n, 1}, 0 )_1$
is constant of rank at most
$d (n + 1) [r (n) - s (n) - t (n)]$.
\item
$\Gamma_{n + 1, n} ( c_{n, 1}, 0 )_2$
is constant of rank at most
$k (n + 1) [r (n) - s (n) - t (n)]$.
\item
$\Gamma_{n + 1, n} ( 0, g_n )_1$
is constant of rank at most
$k (n + 1) t (n)$.
\item
$\Gamma_{n + 1, n} ( 0, g_n )_2$
is constant of rank at most
$d (n + 1) t (n)$.
\end{itemize}
Putting these together,
we get in the first coordinate of
$\Gamma_{n + 1, n} (b_n)$ the direct sum of
$c_{n + 1, 0}$ as described
and a constant function of rank at most
\[
d (n + 1) [r (n) - s (n) - t (n)] + k (n + 1) t (n) \, .
\]
A computation shows that this expression is
equal to $r (n + 1) - s (n + 1) - t (n + 1)$.
In the second coordinate we get a
constant function of rank at most
\[
k (n + 1) s (n) + k (n + 1) [r (n) - s (n) - t (n)] + d (n + 1) t (n)
 = t (n + 1) \, .
\]
This completes the induction,
and the proof.
\end{proof}

\begin{Cor}\label{C_7808_BigRank}
Adopt the assumptions and notation of Notation~\ref{N_7806_Bott}.
Let $n \in \Nz$.
Let $e = (e_1 ,e_2)$ be an element in
$M_{\infty} (C_n) \cong M_{\infty}(C (X_n) \oplus C (Y_n))$
such that $e_1$ is a projection
which is equivalent to a constant projection.
If there exists $x \in M_{\infty} (C_n)$
such that $\| x e x^* - b_n \| < \frac{1}{2}$ then
$\rank (e_1) \geq 2 s (n)$.
\end{Cor}

\begin{proof}
Recall from
Construction
\ref{Cn_6918_General_Part_1a}(\ref{Cn_6918_General_1a_Xn})
and Notation~\ref{N_7730_Cone}
that
\[
Z_n = (S^2)^{s (n)}
\andeqn
Z_n \subset \cone (Z_n) = X_n \subset X_n \amalg Y_n \, .
\]
Also recall the line bundle $L$
and the projection~$p$ from Notation~\ref{N_7806_Bott}.

It follows from Lemma~\ref{L_7906_RankB}
that there is a \pj{} $q \in M_{2 r (n)} (C (Z_n))$
whose range is isomorphic to the section space
of the $s (n)$-dimensional vector bundle
$L^{\times s (n)}$
and such that $q (b_n |_{Z_n} ) q = q$.
Now $\| x e x^* - b_n \| < \frac{1}{2}$
implies $\| q (x e x^* |_{Z_n} ) q - q \| < \frac{1}{2}$.
Since $e |_{Z_n}$ and $q |_{Z_n}$ are projections,
it follows that $q |_{Z_n}$
is Murray-von Neumann equivalent to
a subprojection of $e |_{Z_n} = e_1 |_{Z_n}$.
Therefore $\rank ( e |_{Z_n} ) \geq 2 s (n)$
by Lemma~\ref{Lemma:Chern-class}.
So $\rank ( e_1 ) \geq 2 s (n)$.
\end{proof}

Although not strictly needed for the sequel, we record the following.

\begin{Cor}\label{Cor-rc-C}
Assume the notation and choices in
parts (\ref{Cn_6918_General_dn}),
(\ref{Cn_6918_General_Size})
(including $\kp > \frac{1}{2}$),
(\ref{Cn_6918_General_Algs}),
(\ref{Cn_6918_General_Maps}), and~(\ref{Cn_6918_General_Limit})
of Construction~\ref{Cn_6918_General},
and in Construction \ref{Cn_6918_General_Part_1a}
(except part~(\ref{Cn_9Y01_Gen_1a_Maps2}))
and the parts of Construction~\ref{Cn_6918_General}
referred to there.
Then the algebra $C$ satisfies  $\rc (C) \geq 2 \kappa - 1 > 0$.
\end{Cor}

\begin{proof}
Suppose $\rh < 2 \kappa - 1$.
We show that $C$ does not have $\rh$-comparison.
Choose $n \in \N$ such that $1 / {r (n)} < 2 \kp - 1 - \rh$.
Choose $M \in \Nz$ such that $\rh + 1 < M / {r (n)} < 2 \kp$.
Let $e \in M_{\infty} (C_n)$ be a trivial \pj{}
of rank~$M$.
By slight abuse of notation,
we use $\tgamma_{m, n}$ to denote the amplified map
from $M_{\infty} (C_n)$ to $M_{\infty} (C_m)$ as well.
For $m > n$, the rank of $\tgamma_{m, n} (e)$
is $M \cdot \frac{r (m)}{r (n)}$,
and the choice of $M$ guarantees that
this rank is strictly less than $2 s (m)$.
Now, for any trace $\tau$ on $C_m$
(and thus for any trace on $C$),
and justifying the last step afterwards, we have
\[
d_{\tau} (\tgamma_{m, n} (e))
 = \tau (\tgamma_{m, n} (e))
 = \frac{1}{r (m)} \cdot M \cdot \frac{r (m)}{r (n)}
 \geq 1 + \rh
 > d_{\tau} (b_m) + \rh \, .
\]
To explain the last step,
recall $b_m$ from Notation~\ref{N_7806_Bott},
and use Lemma~\ref{L_7906_RankB}
to see that the ranks of its components
$(b_m)_1 \in M_2 \bigl( C (X_m, M_{r (m)}) \bigr)$
and $(b_m)_2 \in M_2 \bigl( C (Y_m, M_{r (m)}) \bigr)$
are both less than $r (m)$,
while the identity element has rank $r (m)$.

On the other hand, if
$\tgamma_{\infty, 0} (b) \precsim \tgamma_{\infty, n} (e)$
then, in particular, there exists some $m > n$
and $x \in M_{\infty} (C_m)$
such that $\|x\tgamma_{m, n} (e)x^* - b_m\| < \frac{1}{2}$,
which contradicts Corollary~\ref{C_7808_BigRank}.
\end{proof}

\begin{Notation}\label{notation-q-q-perp}
We assume the notation and choices in
parts (\ref{Cn_6918_General_dn}),
(\ref{Cn_6918_General_Spaces}),
(\ref{Cn_6918_General_Algs}),
(\ref{Cn_6918_General_Maps}),
and~(\ref{Cn_6918_General_Limit})
of Construction~\ref{Cn_6918_General}.
In particular, $C_0 = C (X_0) \oplus C (Y_0)$.
Define $q_0 = (1, 0) \in C (X_0) \oplus C (Y_0)$
and $q_0^{\perp} = 1 - q_0$.
For $n \in \N$
define $q_n = \tgamma_{n, 0} (q_0) \in C_n$
and $q_n^{\perp} = 1 - q_n$,
and finally,
define $q = \tgamma_{\infty, 0} (q_0) \in C$
and $q^{\perp} = 1 - q$.
\end{Notation}

\begin{Lemma}\label{L_7816_Rankq}
Make the assumptions in Notation~\ref{notation-q-q-perp}.
Further assume the notation and choices in
Construction \ref{Cn_6918_General_Part_1a}
(except part~(\ref{Cn_9Y01_Gen_1a_Maps2})).
Then the projection
\[
1 - q_n \in M_{l (n)} ( C (X_n) ) \oplus M_{l (n)} ( C (Y_n) )
\]
has the form $(e, f)$
for a constant projection
$e \in M_{l (n)} ( C (X_n) ) = C (X_n, M_{l (n)})$
of rank $t (n)$ and a constant projection
$f \in M_{l (n)} ( C (Y_n) ) = C (Y_n, M_{l (n)})$
of rank $r (n) - t (n)$.
\end{Lemma}

{}From Construction \ref{Cn_6918_General_Part_1a},
we don't actually need to know anything about
the spaces $X_n$ and~$Y_n$,
we don't need to know anything about the points $x_n$
and~$y_n$ except which spaces they are in,
and we don't need to know anything about the
maps $Q_j^{(n)}$ and $R_{n, j}$ except their domains
and codomains.

\begin{proof}[Proof of Lemma~\ref{L_7816_Rankq}]
The proof is an easy induction argument,
using the fact that the image of a constant function
under a diagonal map is again a constant function.
\end{proof}

\begin{Lemma}\label{Lemma:rc_lower_bound}
Assume the notation and choices in
parts (\ref{Cn_6918_General_dn})--(\ref{Cn_6918_General_Limit})
of Construction \ref{Cn_6918_General},
Construction
\ref{Cn_6918_General_Part_1a}
(except part~(\ref{Cn_9Y01_Gen_1a_Maps2})),
and Notation~\ref{notation-q-q-perp},
including
$k (n) < d (n)$
for all $n \in \Nz$,
$\kp > \frac{1}{2}$,
$\om > \om'$,
and
$2 \kappa - 1 > 2 \om$.
Then
\[
\rc (q^{\perp} C q^{\perp}) \geq \frac{2 \kappa - 1}{2 \om} \, .
\]
\end{Lemma}

\begin{proof}
We proceed as in the proof of Corollary~\ref{Cor-rc-C},
although the rank computations are somewhat more involved.
The difference is in the definition of $d_{\tau}$.
In this corner,
$d_{\tau}$ is normalized so that $d_{\tau} (q^{\perp}) = 1$
for all $\tau \in \T (C)$.
To avoid redefining the notation,
we will use $\tau$ to denote a tracial state on~$C$,
and therefore our dimension functions will be of the form
$a \mapsto d_{\tau} (a)/\tau (q^{\perp})$,
noting that $\tau (q^{\perp}) = d_{\tau}(q^{\perp})$
since $q^{\perp}$ is a projection.

It suffices to show that for all
$\rh \in \left (1, \frac{2 \kappa - 1}{2 \om} \right ) \cap \Q$,
we have $\rc (q^{\perp} C q^{\perp}) \geq \rh$.

Fix $\delta \in (0, \om)$ such that
\begin{equation}\label{Eq_7906_Star}
\rh < (1 - \delta) \left( \frac{2 \kappa - 1}{2 \om} \right) \, .
\end{equation}
Set
\begin{equation}\label{Eq_7906_DefEps}
\eps = \frac{\delta}{2 \rho (1 - \delta)} > 0.
\end{equation}
Since the sequence
$\left (\frac{s (n)}{r (n)} \right)_{n = 0, 1, 2, \ldots}$
is nonincreasing and converges to a nonzero limit~$\kappa$,
there exists $n_0 \in \Nz$
such that for all $n$ and $m$ with $m\geq n \geq n_0$, we have
\[
0 \leq 1 - \frac{r (n)}{s (n)} \cdot \frac{s (m)}{r (m)}
 < \eps
\, .
\]
This implies that
\begin{equation}\label{Eq_7906_epsEst}
\frac{r (m)}{r (n)} - \frac{s (m)}{s (n)}
 < \eps \cdot \frac{r (m)}{r (n)}
\, .
\end{equation}
Using~(\ref{Eq_7906_Star}) and
$\delta < \om$ at the first step,
we get
\[
1 - \om + 2 \rh \om
 < 1 - \dt
    + 2 (1 - \delta) \left( \frac{2 \kappa - 1}{2 \om} \right) \om
 = 2 \kp (1 - \dt) \, .
\]

Now write $\rh = \af / \bt$
with $\af, \bt \in \N$.
Choose $n \geq n_0$ such that
\[
\frac{\bt}{r (n)} < 2 \kp (1 - \dt) - (1 - \om + 2 \rh \om) \, .
\]
Then there exists $N_1 \in \N$
such that $\rh N_1 \in \N$ and
\begin{equation}\label{Eq_7906_DefM1}
2 \kappa (1 - \delta) > \frac{N_1}{r (n)} > 1 - \om + 2 \rh \om
\, .
\end{equation}
Set
\begin{equation}\label{Eq_8Y28_NL}
N_2 = \rh N_1 \, .
\end{equation}
Using $\rh > 1$ at the last step,
we have
\[
\frac{N_2}{r (n)}
 = \frac{\rh N_1}{r (n)}
 > \rh (1 - \om + 2 \rh \om)
 > \rh (1 - \om) + 2 \om \, .
\]

Now suppose
$e \in M_{\infty} (C_n)
   = M_{\infty} \bigl( C (X_n) \oplus C (Y_n) \bigr)$
is an ordered pair whose first component
is a trivial projection on $X_n$ of rank $N_1$
and whose second component
is a (trivial) projection on $Y_n$ of rank $N_2$.
Let $m > n$,
and let $f$ be the first component of $\tgamma_{m, n} (e)$;
we estimate $\rank (f)$.
(The
second component is a trivial projection over $Y_m$
whose rank we don't care about.)
Now $f$ is the direct sum of $r (m) / r (n)$ trivial projections,
coming from $C (X_n, M_{r (n)})$ and $C (Y_n, M_{r (n)})$.
At least $s (m) / s (n)$ of these summands come from
$C (X_n, M_{r (n)})$.
So at most $r (m) / r (n) - s (m) / s (n)$
of these summands come from
$C (Y_n, M_{r (n)})$.
The summands coming from $C (X_n, M_{r (n)})$ have rank~$N_1$
and the summands coming from $C (Y_n, M_{r (n)})$ have rank~$N_2$.
Since $N_2 > N_1$,
we get
\begin{align*}
\rank (f)
& \leq \left( \frac{r (m)}{r (n)} - \frac{s (m)}{s (n)} \right) N_2
       + \frac{s (m)}{s (n)} \cdot N_1
\\
& = \frac{r (m)}{r (n)} \cdot N_1
     + \left (\frac{r (m)}{r (n)} - \frac{s (m)}{s (n)} \right )
       (N_2 - N_1) \, .
\end{align*}
Combining this with~(\ref{Eq_7906_epsEst}) at the first step,
and using (\ref{Eq_8Y28_NL}) at the second step,
(\ref{Eq_7906_DefEps}) at the third step,
(\ref{Eq_7906_DefM1}) at the fifth step,
and
Construction \ref{Cn_6918_General}(\ref{Cn_6918_General_Size})
at the sixth step,
we get
\begin{align*}
\rank (f)
& < \frac{r (m)}{r (n)} \cdot (N_1 + \eps N_2)
  = \frac{r (m)}{r (n)} \cdot (1 + \eps \rh) \cdot N_1
\\
& = \frac{r (m)}{r (n)} \cdot \frac{2 - \dt}{2 (1 - \dt)} \cdot N_1
  < \frac{r (m)}{r (n)} \cdot \frac{N_1}{1 - \delta}
  < 2 \kp r (m)
  \leq 2 s (m)
\, .
\end{align*}
So Corollary~\ref{C_7808_BigRank}
implies that there is no $x \in M_{\infty} (C_m)$
for which $\|x \tgamma_{n, m} (e) x^* - b_m \| < \frac{1}{2}$.
Since $m > n$ is arbitrary,
\begin{equation}\label{Eq_8Y29_2NL}
\tgamma_{\infty, n} (e)  \not\precsim b
\, .
\end{equation}

Now let $\tau$ be a trace on $C$,
and restrict it to
$C_n \cong M_{r (n)} \big( C (X_n) \oplus C (Y_n) \big)$.
Denote by $\tr$ the normalized trace on $M_{r (n)}$.
There is a probability measure $\mu$ on $X_n \amalg Y_n$
such that $\tau (a) = \int_{X_n \amalg Y_n} \tr (a) \, d \mu$
for all $a \in C_n$.
Define $\lambda = \mu (X_n)$, so $1 - \lambda = \mu (Y_n)$.
Then, using (\ref{Eq_8Y28_NL}) at the second step,
\[
\tau (e)
 = \frac{\lambda N_1 + (1 - \lambda) N_2}{r (n)}
 = \frac{[\lambda + \rh (1 - \lambda)] N_1}{r (n)}
\, .
\]

Using Lemma~\ref{L_7816_Rankq}
to calculate the ranks of the components of $q_n^{\perp}$,
we get
\begin{equation}\label{Eq_7910_tauqnperp}
\tau (q_n^{\perp})
 = \frac{\lambda t (n) + (1 - \lambda) [r (n) - t (n)]}{r (n)}
\end{equation}
and
\begin{equation}\label{Eq_7910_tauqn}
\tau (q_n)
 = 1 - \tau (q_n^{\perp})
 = \frac{\lambda [r (n) - t (n)] + (1 - \lambda) t (n)}{r (n)}
\, .
\end{equation}
It follows from Lemma~\ref{L_7906_RankB}
and Lemma~\ref{L_7816_Rankq}
that $d_{\tau} (b_n) \leq \tau (q_n)$.
Using this at the first step,
and (\ref{Eq_7910_tauqnperp}) and~(\ref{Eq_7910_tauqn})
at the second step,
we get
\[
\frac{d_{\tau} (b_n)}{\tau (q_n^{\perp})}
 \leq \frac{\tau (q_n)}{\tau (q_n^{\perp})}
 = \frac{\lambda [r (n) - t (n)]
   + (1 - \lambda) t (n)}{\lambda t (n) + (1 - \lambda) [r (n) - t (n)]}
\, .
\]
So
\[
\frac{\tau (e) -  d_{\tau} (b_n)}{\tau (q_n^{\perp})}
 \geq \frac{\big( \lambda + \rh (1 - \lambda) \big) N_1
  - \big( \lambda [r (n) - t (n)]
   + (1 - \lambda) t (n) \big)}{
  \lambda t (n) + (1 - \lambda) [r (n) - t (n)]}
\, .
\]

The last expression is a fractional linear function in
$\lambda$,
and is defined for all values of $\lambda$ in the interval $[0, 1]$.
Any such function is monotone on $[0, 1]$.
In the following calculations,
we recall from Lemma~\ref{L_6922_Rho}
that $\omega \leq \frac{t (n)}{r (n)} < 2 \omega$.
If we set $\lambda = 1$ and use~(\ref{Eq_7906_DefM1}),
the value we obtain is
\[
\frac{N_1 / r (n) - (1 - t (n)/r (n))}{t (n)/r (n)}
 > \frac{(1 - \om + 2 \rh \om)
        - (1 - \om)}{2 \om}
 = \rh \, .
\]
If we set $\lambda = 0$, we get,
using~(\ref{Eq_7906_DefM1}) at the first step
and $\rh > 1$ at the last step,
\[
\frac{\rh N_1/r (n) - t (n)/r (n)}{1 - t (n)/r (n)}
 > \frac{\rh (1 - \om + 2 \rh \om) - 2 \om}{
    1 - \om}
 = \rh + \frac{2 \rh^2 \om - 2 \om}{1 - \om} > \rh \, .
\]
Therefore
\[
\frac{d_{\tau} (\tgamma_{\I, n} (e) )}{d_{\tau} (q^{\perp})}
 > \frac{d_{\tau} (b)}{d_{\tau} (q^{\perp})} + \rh
\]
for all traces $\tau$ on $C$,
so $\rc (q^{\perp} C q^{\perp}) > \rh$, as required.
\end{proof}

We now turn to the issue of finding upper bounds
on the radius of comparison.
For this, we appeal to results of Niu from \cite{niu-mean-dimension}.
Niu introduced a notion of mean dimension
for a diagonal AH-system, \cite[Definition 3.6]{niu-mean-dimension}.
Suppose we are given a direct system
of homogeneous algebras of the form
\[
A_n = C (K_{n, 1}) \otimes M_{j_{n, 1}}
  \oplus C (K_{n, 2}) \otimes M_{j_{n, 2}} \oplus
   \cdots \oplus C (K_{n, m (n)}) \otimes M_{j_{n, m (n)}} \, ,
\]
in which each of the spaces involved is a connected finite CW~complex,
and the connecting maps are unital diagonal maps.
Let $\gamma$ denote the mean dimension of this system,
in the sense of Niu.
It follows trivially from \cite[Definition 3.6]{niu-mean-dimension} that
\[
\gamma
 \leq \lim_{n \to \infty} \max
  \left (\left \{ \frac{\dim (K_{n, l})}{j_{n, l}}
    \mid l = 1, 2, \ldots, m (n) \right \} \right ) \, .
\]
Theorem~6.2 of \cite{niu-mean-dimension} states that
if $A$ is the direct limit of a system as above,
and $A$ is simple,
then $\rc (A) \leq \gamma/2$.
Since the system we are considering here is of this type,
Niu's theorem applies.
With that at hand, we can derive an upper bound
for the radius of comparison of the complementary corner.

\begin{Lemma}\label{Lemma:rc-upper-bound}
Under the same assumptions as in Lemma~\ref{Lemma:rc_lower_bound},
we have
\[
\rc (q C q) \leq \frac{1}{1 - 2 \om} \, .
\]
\end{Lemma}

\begin{proof}
The algebra $C$ is simple by Lemma~\ref{L_0725_Simplicity},
so $q C q$ is also simple.
This fact and Lemma~\ref{L_7816_Rankq}
allow us to apply the discussion above,
getting
\[
\rc (q C q)
 \leq \frac{1}{2} \lim_{n \to \infty}
   \max \left ( \frac{\dim (X_n)}{\rank (q_n |_{X_n})}, \,
    \frac{\dim (Y_n)}{\rank ( q_n |_{Y_n} )} \right )\, .
\]
As $\dim (Y_n) = 1$ for all $n$, the second term converges to~$0$.
As for the first term, by
Construction
\ref{Cn_6918_General_Part_1a}(\ref{Cn_6918_General_1a_Xn}),
we have
$\dim (X_n) = 2 s (n) + 1$.
Also, $\rank (q_n |_{X_n}) = r (n) - t (n)$
by Lemma~\ref{L_7816_Rankq}.
Thus,
using Construction \ref{Cn_6918_General}(\ref{Cn_6918_General_dn}),
Lemma~\ref{L_6922_Rho},
and $d (n) \to \I$
(which follows from
Construction \ref{Cn_6918_General}(\ref{Cn_6918_General_rrp}))
at the last step,
\[
\lim_{n \to \infty} \frac{\dim (X_n)}{\rank (q_n |_{X_n})}
 = \lim_{n \to \infty} \frac{2 s (n) + 1}{r (n) - t (n)}
 \leq \lim_{n \to \infty}  \frac{2 r (n) + 1}{r (n) - t (n)}
 \leq \frac{2}{1 - 2 \om}\, .
\]
This gives us the required estimate.
\end{proof}

\begin{Lemma}\label{lemma-stable-rank-1}
Let the assumptions and notation be as in
Notation~\ref{notation-q-q-perp},
Construction
\ref{Cn_6918_General_Part_1a}(\ref{Cn_6918_General_1a_Xn}),
and Construction
\ref{Cn_6918_General_Part_1a}(\ref{Cn_6918_General_1a_YnIs01}).
If $e \in C$ is a projection which has the same same $K_0$-class as $q$
then $e$ is unitarily equivalent to $q$.
The same holds with $q^{\perp}$ in place of $q$.
\end{Lemma}

\begin{proof}
This can be seen directly from the construction.
For each $n \in \Nz$,
since $X_n$ and $Y_n$
are contractible
(Construction
\ref{Cn_6918_General_Part_1a}(\ref{Cn_6918_General_1a_Xn})
and Construction
\ref{Cn_6918_General_Part_1a}(\ref{Cn_6918_General_1a_YnIs01})),
if $e \in M_{\infty} (C_n)$ is a projection which
has the same $K_0$-class as $q$,
then $e$ is actually unitarily equivalent to $q_n$.
The same holds for $q_n^{\perp}$.
It follows that this is the case in $C$ as well.
\end{proof}

We point out that this lemma can also be deduced using cancellation.
By \cite[Theorem 4.1]{elliott-ho-toms},
simple unital AH algebras
which arise from AH systems with diagonal maps
have stable rank~$1$.
Rieffel has shown that $C^*$-algebras
with stable rank $1$ have cancellation;
see \cite[Theorem 6.5.1]{blackadar-k-theory}.

\section{The tracial state space}\label{Sec_Traces}

For a compact Hausdorff space~$X$,
we will need all of $C (X, \R)$
(the space of real valued \cfn{s} on~$X$),
the tracial state space of $C (X)$
(and of $C (X, M_n)$),
and the space of affine functions on the tracial state space.
This last space is an order unit space,
and much of our work will be done there.

For later reference,
we recall some of the definitions,
and then describe how to move between these spaces.
We begin with the definition of an order unit space
from the discussion before Proposition II.1.3 of~\cite{Alf}.
We suppress the order unit in our notation,
since (except in several abstract results)
our order unit spaces will always be sets of
affine \cfn{s} on compact convex sets with order unit
the constant function~$1$.

\begin{Def}\label{D_6919_OrUnitSp}
An {\emph{order unit space}}~$V$
is a partially ordered real Banach space
(see page~1 of~\cite{Gdr0} for the axioms of a
partially ordered real vector space)
which is {\emph{Archimedean}}
(if $v \in V$ and $\{ \ld v \mid \ld \in (0, \infty) \}$
has an upper bound, then $v \leq 0$),
with a distinguished element $e \in V$
which is an {\emph{order unit}}
(that is,
for every $v \in V$ there is $\ld \in (0, \infty)$
such that $- \ld e \leq v \leq \ld e$),
and such that the norm on~$V$ satisfies
\[
\| v \|
= \inf \big(
 \big\{ \ld \in (0, \infty) \mid
     - \ld e \leq v \leq \ld e \big\} \big)
\]
for all $v \in V$.

The morphisms of order unit spaces are
the positive linear maps which preserve the order units.
\end{Def}

The morphisms of compact convex sets
(compact convex subsets of locally convex topological vector spaces)
are just the continuous affine maps.

\begin{Def}\label{N_7701_Aff}
If $K$ is a compact convex set,
we denote by $\Aff (K)$
the order unit space of \ct{} affine functions $f \colon K \to \R$,
with the supremum norm
and with order unit the constant function~$1$.

If $K$ and $L$ are compact convex sets
and $\ld \colon K \to L$
is continuous and affine,
we let $\ld^* \colon \Aff (L) \to \Aff (K)$
be the positive linear order unit preserving map
given by $\ld^* (f) = f \circ \ld$
for $f \in \Aff (L)$.
\end{Def}

This definition makes $K \mapsto \Aff (K)$ a functor.

\begin{Def}\label{N_7701_StateSpace}
If $V$ is an order unit space with order unit~$e$,
we denote by $S (V)$ (or $S (V, e)$ if $e$ is not understood)
its state space
(the order unit space morphisms to $(\R, 1)$),
which is a compact convex set with the weak* topology.

If $W$ is another order unit space and
$\ph \colon V \to W$
is positive, linear, and order unit preserving,
we let $S (\ph) \colon S (W) \to S (V)$
be the continuous affine map
given by $S (\ph) (\om) = \om \circ \ph$ for $\om \in S (W)$.
\end{Def}

This definition makes $V \mapsto S (V)$ a functor.

\begin{Thm}[Theorem~7.1 of~\cite{Gdr0}]\label{T_7701_AffKToK}
There is a natural isomorphism $S (\Aff (K)) \cong K$
for compact convex sets~$K$,
given by sending $x \in K$ to the evaluation map
$\ev_x \colon \Aff (K) \to \R$
defined by $\ev_x (f) = f (x)$ for $f \in \Aff (K)$.
\end{Thm}

\begin{Def}\label{N_6919_Traces}
For a unital C*-algebra~$A$,
we denote its tracial state space by $\T (A)$.

If $A$ and $B$ are unital \ca{s}
and $\ph \colon A \to B$ is a unital \hm,
we let $\T (\ph) \colon \T (B) \to \T (A)$
be the continuous affine map
given by $\T (\ph) (\ta) = \ta \circ \ph$
for $\ta \in \T (B)$.
We let ${\widehat{\ph}} \colon \Aff (\T (A)) \to \Aff (\T (B))$
be the positive order unit preserving map given by
${\widehat{\ph}} (f) = f \circ \T (\ph)$
for $f \in \Aff (\T (A))$.
(Thus, ${\widehat{\ph}} = \T (\ph)^*$.)
\end{Def}

\begin{Lemma}\label{L_7701_XToTCX}
Let $X$ be a compact Hausdorff space.
Then $C (X, \R)$,
with the supremum norm
and distinguished element the constant function~$1$,
is a complete order unit space.
Restriction of tracial states on $C (X)$
is an affine homeomorphism
from $\T (C (X))$ to $S (C (X, \R))$.
The map from $X$ to $S (C (X, \R) )$
which sends $x \in X$ to the point evaluation
$\ev_x \colon C (X, \R) \to \R$
is a homeomorphism onto its image,
and the map $R_X \colon \Aff \big( S (C (X, \R)) \big) \to C (X, \R)$,
given by $R_X (f) (x) = f (\ev_x)$
for $f \in \Aff \big( S (C (X, \R)) \big)$ and $x \in X$,
is an isomorphism of order unit spaces.

If $Y$ is another compact Hausdorff space,
then the function which sends a positive linear order unit preserving
map $Q \colon C (X, \R) \to C (Y, \R)$
to $S (Q) \colon S (C (Y, \R)) \to S (C (X, \R))$,
as in Definition~\ref{N_7701_StateSpace},
is a bijection
to the continuous affine maps from $S (C (Y, \R))$ to $S (C (X, \R))$.
Its inverse is the map $E$ given as follows.
For a continuous affine map
$\ld \colon S (C (Y, \R)) \to S (C (X, \R))$,
using the notation of Definition~\ref{N_7701_Aff},
define $E (\ld) \colon C (X, \R) \to C (Y, \R)$
by $E (\ld) = R_Y \circ \ld^* \circ R_X^{-1}$.
\end{Lemma}

A positive linear order unit preserving map
from $C (X, \R)$ to $C (Y, \R)$
is called a {\emph{Markov operator}}.

\begin{proof}[Proof of Lemma~\ref{L_7701_XToTCX}]
It is immediate that $C (X, \R)$ is a complete order unit space.
The identification of $S (C (X, \R))$ is also immediate.
The fact that $R_X$ is bijective
follows from \cite[Corollary 11.20]{Gdr0}
using the identification of $X$
with the extreme points of $S (C (X, \R))$.

For the second paragraph, it is immediate
that $S$ sends positive linear order unit preserving maps
to continuous affine maps,
and that $E$ does the reverse.
For the rest,
we must show that $S \circ E$ and $E \circ S$
are the identity maps on the appropriate sets.

We first claim that for $g \in \Aff \big( S (C (X, \R)) \big)$
and $\rh \in S (C (X, \R))$ we have
\begin{equation}\label{Eq_9X30_MoveRX}
g (\rh) = \rh (R_X (g)) \, .
\end{equation}
This formula is true by definition when $\rh = \ev_x$ for some
$x \in X$.
Since, for fixed~$g$,
both sides of~(\ref{Eq_9X30_MoveRX}) are continuous affine
functions of~$\rh$,
and since $S (C (X, \R))$ is the closed convex hull
of $\{ \ev_x \mid x \in X \}$,
the claim follows.

We next claim that if
$\ld \colon S (C (Y, \R)) \to S (C (X, \R))$
is \ct{} and affine,
$\om \in S (C (Y, \R))$,
and $g \in \Aff \big( S (C (X, \R)) \big)$,
then
\begin{equation}\label{Eq_9X30_MoveLd}
(\om \circ R_Y) (g \circ \ld) = (\ld (\om) \circ R_X) (g) \, .
\end{equation}
To prove this claim,
for the same reasons as in the proof of the first claim,
it suffices to prove this when there is $y \in Y$
such that $\om = \ev_y$.
In this case,
using the definition of $R_Y$ at the second step
and the previous claim with $\rh = \ld (\ev_y)$ at the third step,
\[
(\ev_y \circ R_Y) (g \circ \ld)
 = R_Y (g \circ \ld) (y)
 = (g \circ \ld) (\ev_y)
 = (\ld (\ev_y) \circ R_X) (g) \, ,
\]
as desired.

Now let $\ld \colon S (C (Y, \R)) \to S (C (X, \R))$
be \ct{} and affine;
we prove that $S (E (\ld)) = \ld$.
Let $\om \in S (C (X, \R))$
and let $f \in C (Y, \R)$.
Working through the definitions gives
\[
S (E (\ld)) (\om) (f) = (\om \circ R_Y) ( R_X^{-1} (f) \circ \ld) \, .
\]
By (\ref{Eq_9X30_MoveLd}) with $g = R_X^{-1} (f)$,
the right hand side is $\ld (\om) (f)$,
as desired.

Finally, let $Q \colon C (X, \R) \to C (Y, \R)$
be a positive linear order unit preserving map;
we show that $E (S (Q)) = Q$.
Let $f \in C (X, \R)$ and let $y \in Y$.
Working through the definitions gives
\[
E (S (Q)) (f) (y) = R_X^{-1} (f) (\ev_y \circ Q) \, .
\]
Applying (\ref{Eq_9X30_MoveRX}) with $g = R_X^{-1} (f)$
and $\rh = \ev_y \circ Q$,
we see that the right hand side is $(\ev_y \circ Q) (f) = Q (f) (y)$.
This proves that $E (S (Q)) = Q$,
and the proof is complete.
\end{proof}

Direct limits of direct systems of order unit spaces
are constructed at the beginning of Section~3 of~\cite{thomsen},
including Lemma~3.1 there.

\begin{Prop}\label{P_6919_CommLim}
Let
$\big( (D_n)_{n = 0, 1, 2, \ldots},
( \ph_{n, m} )_{0 \leq m \leq n} \big)$
be a direct system of unital \ca{s}
and unital \hm{s}.
Set $D = \dirlim_{n} D_n$.
Then there are a natural homeomorphism
\[
\T (D) \to \invlim_n \T (D_n)
\]
and a natural isomorphism
\[
\Aff ( \T (D)) \to \dirlim_n \Aff (\T (D_n))
\]
of order unit spaces.
\end{Prop}

\begin{proof}
The first part is Lemma~3.3 of~\cite{thomsen}.

The second part is Lemma~3.2 of~\cite{thomsen},
combined with the fact (Theorem~\ref{T_7701_AffKToK})
that the state
space of $\Aff (K)$ is naturally identified with $K$.
\end{proof}

\begin{Def}\label{D_6921_DSumOrdUn}
Let $V$ and $W$ be order unit spaces,
with order units $e \in V$ and $f \in W$.
We define the
{\emph{direct sum}}
$V \oplus W$ to be
the vector space direct sum $V \oplus W$ as a real vector space,
with the order $(v_1, w_1) \leq (v_2, w_2)$
for $v_1, v_2 \in V$ and $w_1, w_2 \in W$
if and only if $v_1 \leq v_2$ and $w_1 \leq w_2$,
with the order unit $(e, f)$,
and the norm $\| (v, w) \| = \max ( \| v \|, \, \| w \| )$.
\end{Def}

\begin{Lemma}\label{L_6921_DSumIsOU}
Let $V$ and $W$ be order unit spaces.
Then $V \oplus W$ as in Definition~\ref{D_6921_DSumOrdUn}
is an order unit space,
which is complete if $V$ and $W$ are.
\end{Lemma}

\begin{proof}
The proof is straightforward.
\end{proof}

\begin{Lemma}\label{L_6921_DSumTraces}
Let $A$ and $B$ be unital \ca{s}.
Then,
taking the direct sum on the right to be
as in Definition~\ref{D_6921_DSumOrdUn},
there is an isomorphism
\[
\Aff (\T (A \oplus B )) \cong \Aff (\T (A)) \oplus \Aff (\T (B )) \, ,
\]
given as follows.
Identify $\T (A)$ with a subset of $\T (A \oplus B )$
by, for $\ta \in \T (A)$,
defining $i (\ta) (a, b) = \ta (a)$
for all $a \in A$ and $b \in B$,
and similarly identify $\T (B)$ with a subset of $\T (A \oplus B )$.
Then the map
$\Aff (\T (A \oplus B )) \to \Aff (\T (A)) \oplus \Aff (\T (B ))$
is $f \mapsto (f |_{\T (A)}, \, f |_{\T (B)})$.
\end{Lemma}

\begin{proof}
It is clear that if $f \in \Aff (\T (A \oplus B))$,
then $f |_{\T (A)} \in \Aff (\T (A))$
and $f |_{\T (B)} \in \Aff (\T (B))$,
and moreover that the map
of the lemma is linear, positive,
and preserves the order units.
One easily checks that every \tst{} on $A \oplus B$
is a convex combination of tracial states on $A$ and~$B$,
from which it follows that if
$f |_{\T (A)} = 0$ and $f |_{\T (B)} = 0$ then $f = 0$.

It remains to prove that the map of the lemma is surjective.
Let $g \in \Aff (\T (A))$ and $h \in \Aff (\T (B))$.
Define
$f \colon \T (A \oplus B ) \to \R$
by, for $\ta \in \T (A \oplus B )$,
\[
f (\ta)
= \ta (1, 0) g \big( \ta (1, 0)^{-1} \ta |_A \big)
+ \ta (0, 1) g \big( \ta (0, 1)^{-1} \ta |_B \big)
\]
(taking the first summand to be zero if $\ta (1, 0) = 0$
and the second summand to be zero if $\ta (0, 1) = 0$).
Straightforward but somewhat tedious calculations
show that $f$ is weak* \ct{} and affine,
and clearly $f |_{\T (A)} = g$ and $f |_{\T (B)} = h$.
\end{proof}

The following result generalizes Lemma~3.4 of~\cite{thomsen}.
It still isn't the most general Elliott approximate intertwining
result for order unit spaces,
because we assume that the underlying order unit spaces
of the two direct systems are the same.
The main effect of this assumption is to simplify the notation.

\begin{Prop}\label{P_6917_OrdULim}
Let $(V_m)_{m = 0, 1, 2, \ldots}$ be a sequence of
separable complete order unit spaces,
and let
\[
\big( (V_m)_{m = 0, 1, 2, \ldots},
\, ( \ph_{n, m} )_{0 \leq m \leq n} \big)
\andeqn
\big( (V_m)_{m = 0, 1, 2, \ldots},
\, ( \ph_{n, m}' )_{0 \leq m \leq n} \big)
\]
be two direct systems of order unit spaces,
using the same spaces,
and with maps $\ph_{n, m}, \ph_{n, m}' \colon V_m \to V_n$
which are linear, positive,
and preserve the order units.
Let $V$ and $V'$ be the direct limits
\[
V = \dirlim_n
\big( (V_m)_{m = 0, 1, 2, \ldots},
\, ( \ph_{n, m} )_{0 \leq m \leq n} \big)
\]
and
\[
V' = \dirlim_n
\big( (V_m)_{m = 0, 1, 2, \ldots},
\, ( \ph_{n, m}' )_{0 \leq m \leq n} \big) \, ,
\]
with corresponding maps
\[
\ph_{\infty, n} \colon V_n \to V
\andeqn
\ph_{\infty, n}' \colon V_n \to V'
\]
for $n \in \Nz$.
For $n \in \Nz$ further let
\[
v_0^{(n)}, \, v_1^{(n)}, \, \ldots
\in V_n
\]
be a dense sequence in the closed unit ball of~$V_n$,
and define $F_n \S V_n$
to be the finite set
\[
F_n = \bigcup_{m = 0}^n
   \left[ \big\{ \ph_{n, m} \big( v_k^{(m)} \big) \colon
    0 \leq k \leq n \big\}
\cup
\big\{ \ph_{n, m}' \big( v_k^{(m)} \big) \colon
    0 \leq k \leq n \big\} \right] \, .
\]
Suppose that there are $\dt_0, \dt_1, \ldots \in (0, \infty)$
such that
\begin{equation}\label{Eq_6917_SumDt}
\sum_{n = 0}^{\infty} \dt_n < \infty
\end{equation}
and for all $n \in \Nz$ and all $v \in F_n$ we have
\[
\| \ph_{n + 1, \, n} (v) - \ph_{n + 1, \, n}' (v) \| < \dt_n \, .
\]
Then there is a unique isomorphism
$\rh \colon V \to V'$
such that for all $m \in \Nz$ and all $v \in V_m$
we have
\[
\rh ( \ph_{\infty, m} (v))
= \lim_{n \to \infty} ( \ph_{\infty, n}' \circ \ph_{n, m} ) (v) \, .
\]
Its inverse is determined by
\[
\rh^{-1} ( \ph_{\infty, m}' (v))
= \lim_{n \to \infty} ( \ph_{\infty, n} \circ \ph_{n, m}' ) (v)
\]
for $m \in \Nz$ and $v \in V_m$.
\end{Prop}

\begin{proof}
We first claim that for $m \in \Nz$ and $v \in F_m$,
the sequence
$\big( ( \ph_{\infty, n}' \circ \ph_{n, m} ) (v) \big)_{n \geq m}$
is a Cauchy sequence in~$V'$.
For $n \geq m$,
we estimate,
using $\| \ph_{\infty, n + 1}' \| \leq 1$, $\| v \| \leq 1$,
and $\ph_{n, m} (v) \in F_n$ at the last step:
\begin{align*}
& \big\| ( \ph_{\infty, n + 1}' \circ \ph_{n + 1, \, m} ) (v)
- ( \ph_{\infty, n}' \circ \ph_{n, m} ) (v) \big\|
\\
& \hspace*{3em} {\mbox{}}
= \big\| ( \ph_{\infty, n + 1}'
\circ \ph_{n + 1, \, n} \circ \ph_{n, m} )
(v)
- ( \ph_{\infty, n + 1}' \circ \ph_{n + 1, \, n}' \circ \ph_{n, m} )
(v) \big\|
\\
& \hspace*{3em} {\mbox{}}
\leq \| \ph_{\infty, n + 1}' \|
\big\| \ph_{n + 1, \, n} ( \ph_{n, m} (v) )
- \ph_{n + 1, \, n}' ( \ph_{n, m} (v) ) \big\|
\leq \dt_n \, .
\end{align*}
The claim now follows from~(\ref{Eq_6917_SumDt}).

Next,
we claim that for $m \in \Nz$ and $k \in \N$,
the sequence
$\big( ( \ph_{\infty, n}' \circ \ph_{n, m} ) \big( v_k^{(m)} \big)
   \big)_{n \geq m}$
is a Cauchy sequence in~$V'$.
Indeed,
taking $m_0 = \max (m, k)$,
this follows from the previous claim and the fact that
$\ph_{m_0, m} \big( v_k^{(m)} \big) \in F_{m_0}$.

Now we claim that for $m \in \Nz$ and $v \in V_m$,
the sequence
$( ( \ph_{\infty, n}' \circ \ph_{n, m} ) (v) )_{n \geq m}$
is a Cauchy sequence in~$V'$.
Without loss of generality
$\| v \| \leq 1$.
This claim follows from a standard $\frac{\ep}{3}$~argument:
to show that
\[
\big\| ( \ph_{\infty, n_1}' \circ \ph_{n_1, m} ) (v)
  - ( \ph_{\infty, n_2}' \circ \ph_{n_2, m} ) (v) \big\| < \ep
\]
for all sufficiently large~$n_1$ and~$n_2$,
choose $k \in \N$ such that
$\big\| v - v_k^{(m)} \big\| < \frac{\ep}{3}$,
and use the previous claim.

Since $V'$ is complete,
it follows that
$\lim_{n \to \infty}
( \ph_{\infty, n}' \circ \ph_{n, m} ) (v)$
exists for all $m \in \Nz$ and $k \in \N$.
Since $\| \ph_{\infty, n}' \circ \ph_{n, m} \| \leq 1$
whenever $m, n \in \Nz$ satisfy $m \leq n$,
it follows that for $m \in \N$ there is a unique
bounded linear map
$\rh_m \colon V_m \to V'$
such that $\| \rh_m \| \leq 1$ and
$\rh_m (v)
= \lim_{n \to \infty} ( \ph_{\infty, n}' \circ \ph_{n, m} ) (v)$
for all $k \in \N$.

It is clear from the construction that
$\rh_n \circ \ph_{n, m} = \rh_m$
whenever $m, n \in \Nz$ satisfy $m \leq n$.
By the universal property of the direct limit,
there is a unique bounded linear map $\rh \colon V \to V'$
such that $\rh \circ \ph_{\infty, m} = \rh_m$ for all $m \in \Nz$.
It is clearly contractive, order preserving,
and preserves the order units,
and is uniquely determined as in the statement of the proposition.

The same argument shows that there is a unique contractive linear
map $\ld \colon V' \to V$
determined in the analogous way.
For all $m \in \Nz$,
we have
\[
\ld \circ \rh \circ \ph_{\infty, m}
= \ld \circ \ph_{\infty, m}'
= \ph_{\infty, m} \, ,
\]
so the universal property of the direct limit
implies $\ld \circ \rh = \id_V$.
Similarly $\rh \circ \ld = \id_{V'}$.
\end{proof}

\begin{Prop}\label{P_6918_Nat}
The isomorphism of Proposition~\ref{P_6917_OrdULim}
has the following naturality property.
Let the notation be as there,
and suppose that,
in addition,
we are given
separable complete order unit spaces $W_n$ for $n \in \Nz$,
direct systems
\[
\big( (W_m)_{m = 0, 1, 2, \ldots},
\, ( \ps_{n, m} )_{0 \leq m \leq n} \big)
\andeqn
\big( (W_m)_{m = 0, 1, 2, \ldots},
\, ( \ps_{n, m}' )_{0 \leq m \leq n} \big)
\]
using the same spaces,
with positive linear order unit preserving maps,
with direct limits $W$ and~$W'$,
and with corresponding maps
\[
\ps_{\infty, n} \colon W_n \to W
\andeqn
\ps_{\infty, n}' \colon W_n \to W'
\]
for $n \in \Nz$.
Also suppose that for $n \in \N$ there is a
sequence
\[
w_0^{(n)}, \, w_1^{(n)}, \, \ldots \in W_n
\]
which is dense in the closed unit ball of~$W_n$,
and that there is
a sequence $(\ep_n)_{n = 0, 1, 2, \ldots}$ in $(0, \infty)$
such that $\sum_{n = 0}^{\infty} \ep_n < \infty$
and,
with
\[
G_n = \bigcup_{m = 0}^n
\left[ \big\{ \ps_{n, m} \big( w_k^{(m)} \big) \mid
0 \leq k \leq n \big\}
\cup
\big\{ \ps_{n, m}' \big( w_k^{(m)} \big) \mid
0 \leq k \leq n \big\} \right] \, ,
\]
for all $n \in \Nz$ and all $w \in G_n$ we have
\[
\| \ps_{n + 1, \, n} (w) - \ps_{n + 1, \, n}' (w) \| < \ep_n \, .
\]
Let $\sm \colon W \to W'$ be the isomorphism of
Proposition~\ref{P_6917_OrdULim}.
Suppose further that we have positive linear order unit preserving maps
$\mu_n, \mu_n' \colon V_n \to W_n$
for $n \in \Nz$
such that
\[
\mu_n \circ \ph_{n, m} = \ps_{n, m} \circ \mu_m
\andeqn
\mu_n' \circ \ph_{n, m}' = \ps_{n, m}' \circ \mu_m'
\]
for all $m, n \in \Nz$ with $m \leq n$.
Let $\mu \colon V \to W$ and $\mu' \colon V' \to W'$
be the induced maps of the direct limits.
Then $\mu' \circ \rh = \sm \circ \mu$.
\end{Prop}

\begin{proof}
By construction,
$\rh \colon V \to V'$
and $\sm \colon W \to W'$
are determined by
\begin{equation}\label{Eq_6919_Star}
\rh ( \ph_{\infty, m} (v))
= \lim_{n \to \infty} ( \ph_{\infty, n}' \circ \ph_{n, m} ) (v)
\end{equation}
for $m \in \Nz$ and $v \in V_m$,
and
\begin{equation}\label{Eq_6919_Star_2}
\sm ( \ps_{\infty, m} (w))
= \lim_{n \to \infty} ( \ps_{\infty, n}' \circ \ps_{n, m} ) (w)
\end{equation}
for $m \in \Nz$ and $w \in W_m$.
Using~(\ref{Eq_6919_Star})
at the first step
and~(\ref{Eq_6919_Star_2}) at the last step,
for $m \in \Nz$ and $v \in V_m$ we therefore have
\begin{align*}
(\mu' \circ \rh) ( \ph_{\infty, m} (v) )
& = \mu' \left(
\lim_{n \to \infty} ( \ph_{\infty, n}' \circ \ph_{n, m} ) (v) \right)
= \lim_{n \to \infty}
(\mu' \circ \ph_{\infty, n}' \circ \ph_{n, m} ) (v)
\\
& = \lim_{n \to \infty}
( \ps_{\infty, n}' \circ \ps_{n, m} \circ \mu_m ) (v)
= (\sm \circ \mu) ( \ph_{\infty, m} (v) ) \, .
\end{align*}
Since
$\bigcup_{m = 0}^{\infty} \ph_{\infty, m} (V_m)$ is dense in~$V$,
the result follows.
\end{proof}

Proposition~\ref{P_arbitrary_trace_space} below
can essentially be extracted
from the proof of Lemma~3.7 of~\cite{thomsen}.
We give here a precise formulation which is needed for our purposes.
The difference between our formulation
and that of~\cite{thomsen}
is that we need more control over the matrix sizes
in the construction.
In the argument,
the following result substitutes for Lemma~3.6 there.

\begin{Lemma}[Based on~\cite{thomsen}]\label{L_7630_ConvComb}
Let $X$ and $Y$ be compact Hausdorff spaces,
with $X$ path connected.
Let $\ld \colon \T (C (Y)) \to \T (C (X))$
be affine and \ct.
Let $E (\ld) \colon C (X, \R) \to C (Y, \R)$
be as in Lemma~\ref{L_7701_XToTCX}.
Then for every $\ep > 0$ and every finite set $F \S C (X, \R)$
there exists $N_0 \in \N$
such that for every $N \in \N$ with $N \geq N_0$
there are \cfn{s} $g_1, g_2, \ldots, g_N \colon Y \to X$
such that for every $f \in F$ we have
\[
\Bigg\| E (\ld) (f)
    - \frac{1}{N} \sum_{j = 1}^N f \circ g_j \Bigg\|_{\I}
 < \ep \, .
\]
\end{Lemma}

\begin{proof}
It suffices to prove the result under the additional assumption
that $\| f \| \leq 1$ for all $f \in F$.

Let $\ep > 0$.
Since $E (\ld)$ is a Markov operator,
Theorem~2.1 of~\cite{thomsen}
provides $n \in \N$,
unital \hm{s} $\ps_1, \ps_2, \ldots, \ps_n \colon C (X) \to C (Y)$,
and $\af_1, \af_2, \ldots, \af_n \in [0, 1]$
with $\sum_{l = 1}^n \af_l = 1$
such that
\[
\left\| E (\ld) (f) - \sum_{l = 1}^n \af_l \ps_l (f) \right\|_{\I}
 < \frac{\ep}{2}
\]
for all $f \in F$.
Note that if $\beta_1, \beta_2, \ldots, \beta_n \in [0, 1]$ satisfy
$\sum_{l = 1}^n | \alpha_l - \beta_l | < \frac{\ep}{2}$ then
\[
\left\| E (\ld) (f) - \sum_{l = 1}^n \beta_l \ps_l (f) \right\|_{\I}
< \ep
\]
for all $f \in F$.
Choose $N_0 \in \N$ such that $N_0 > 4 n / \ep$.
Let $N \in \N$ satisfy $N \geq N_0$.
For $l = 1, 2, \ldots, n - 1$
choose
$\bt_l \in \left( \af_l - \frac{1}{N}, \, \af_l \right]
     \cap \frac{1}{N} \Z$,
and set $\bt_n = 1 - \sum_{l = 1}^{n - 1} \bt_l$.
Then
\[
\beta_1, \beta_2, \ldots, \beta_n \in \frac{1}{N} \Nz,
\qquad
\sum_{l = 1}^n \bt_l = 1,
\andeqn
\sum_{l = 1}^n | \alpha_l - \beta_l | < \frac{\ep}{2} \, .
\]

Set $m_l = N \bt_l$ for $l = 1, 2, \ldots, n$.
Then for all $f \in F$
we have
\[
\left\| E (\ld) (f)
    - \frac{1}{N} \sum_{l = 1}^n m_l \ps_l (f) \right\|_{\I}
 < \ep \, .
\]

Now for $l = 1, 2, \ldots, n$
let $h_l \colon Y \to X$ be the \cfn{}
such that $\ps_l (f) = f \circ h_l$
for all $f \in C (X)$,
and for $j = 1, 2, \ldots, N$
define $g_j = h_l$ when
\[
\sum_{k = 1}^{l - 1} m_k < j \leq \sum_{k = 1}^{l} m_k \, .
\]
Then
\[
\frac{1}{N} \sum_{l = 1}^n m_l \ps_l (f)
 = \frac{1}{N} \sum_{j = 1}^N f \circ g_j
\]
for all $f \in C (X)$.
\end{proof}

\begin{Prop}\label{P_arbitrary_trace_space}
Let $K$ be a metrizable Choquet simplex,
and let $(l (n))_{n = 0, 1, 2, \ldots}$ be a sequence of
integers such that $l (n) \geq 2$ for all $n > 0$.
For $n \in \Nz$ set $r (n) = \prod_{j = 1}^n l (j)$.
Then there exist $n_0 < n_1 < n_2 < \cdots \in \N$,
with $n_0 = 0$ and $n_1 = 1$,
and a direct system
\[
C ([0, 1]) \otimes M_{r (n_0)}
 \overset{\alpha_{1, 0}}{\longrightarrow} C ([0, 1]) \otimes M_{r (n_1)}
 \overset{\alpha_{2, 1}}{\longrightarrow} C ([0, 1]) \otimes M_{r (n_2)}
 \overset{\alpha_{3, 2}}{\longrightarrow} \cdots
\]
with injective maps
which are diagonal
(in the sense analogous to
Construction \ref{Cn_6918_General}(\ref{Cn_6918_General_Diag}))
and such that the direct limit $A$ satisfies $\T (A) \cong K$.
\end{Prop}

It is easy to arrange that the algebra $A$ in this proposition
be simple:
by Proposition~\ref{P_6917_OrdULim},
replacement of a small enough fraction of the maps $g_{k, l}$
in the proof with suitable
point evaluations does not change the tracial state space.
However, doing so at this stage does not help with later work.

The conditions $n_0 = 0$ and $n_1 = 1$ are needed because we will
later need to pass to a corresponding subsystem of
a system as in Construction~\ref{Cn_6918_General}
(more accurately, Construction~\ref{Cn_9104_Half} below),
and we want to avoid later complexity of the argument
by preserving the value of~$\om$.

\begin{proof}[Proof of Proposition~\ref{P_arbitrary_trace_space}]
We mostly follow the proof of Lemma~3.7 of~\cite{thomsen},
using Lemma~\ref{L_7630_ConvComb}
in place of Lemma~3.6 of~\cite{thomsen},
and slightly changing the order of the steps
to accommodate the difference
between our conclusion and that of Theorem~3.9 of~\cite{thomsen}.
For convenience,
we will use Proposition~\ref{P_6917_OrdULim}
in place of Lemma~3.4 of~\cite{thomsen}.

For convenience of notation,
and following~\cite{thomsen},
set $P = \T (C ([0, 1]))$.
Lemma 3.8 of~\cite{thomsen}
provides an inverse system
$\big( (P_k)_{k = 0, 1, \ldots}, \,
  (\ld_{j, k})_{0 \leq j \leq k} \big)$
with \ct{} affine maps $\ld_{j, k} \colon P_k \to P_j$ such that
$P_k = P$ for all $k \in \Nz$
and
\begin{equation}\label{Eq_7822_NN}
\varprojlim
 \big( (P_k)_{k = 0, 1, \ldots}, \, (\ld_{j, k})_{0 \leq j \leq k} \big)
 \cong K.
\end{equation}

Choose $f_0, f_1, \ldots \in C ([0, 1], \, \R)$
such that $\{ f_0, f_1, \ldots \}$ is dense in $C ([0, 1], \, \R)$.

We now construct numbers $n_k \in \N$ for $k \in \Nz$,
finite subsets $F_k \S C ([0, 1], \, \R)$ for $k \in \Nz$,
positive unital linear maps
$\ps_{k + 1, \, k} \colon C ([0, 1], \, \R) \to C ([0, 1], \, \R)$
for $k \in \N$,
and \cfn{s}
\[
g_{k, 1}, g_{k, 2}, \ldots, g_{k, \, r (n_{k + 1}) / r (n_k)}
    \colon [0, 1] \to [0, 1] \, ,
\]
such that the following conditions are satisfied:
\begin{enumerate}
\item\label{7701_Ind_F}
$F_0 = \{ f_0 \}$
and for $k \in \Nz$,
\[
F_{k + 1}
= F_k \cup \{ f_{k + 1} \}
    \cup E (\ld_{k, \, k + 1} ) ( F_k \cup \{ f_{k + 1} \} )
    \cup \ps_{k + 1, \, k} ( F_k \cup \{ f_{k + 1} \} ) \, .
\]
\item\label{7701_Ind_Ratio}
$n_0 = 0$, $n_1 = 1$, and $n_2 = 2$,
and for $k \in \N$ with $k \geq 2$ we have
$n_{k + 1} > n_k$
and $r (n_{k + 1}) / r (n_k) > 2^k$.
\item\label{7701_Ind_psi}
For $k \in \Nz$ and $f \in C ([0, 1], \, \R)$,
\[
\ps_{k + 1, \, k} (f)
 = \frac{r (n_k)}{r (n_{k + 1})}
    \sum_{l = 1}^{r (n_{k + 1}) / r (n_k)} f \circ g_{k, l} \, .
\]
\item\label{7701_Ind_Small}
$\big\| E (\ld_{k, \, k + 1} ) (f)
   - \ps_{k + 1, \, k} (f) \big\| < 2^{- k}$
for $k \geq 2$ and $f \in F_k$.
\end{enumerate}

We carry out the construction by induction on~$k$.
Define $F_0 = \{ f_0 \}$, $n_0 = 0$, and $n_1 = 1$.
Take $g_{0, l} \colon [0, 1] \to [0, 1]$
to be the identity map
for $l = 1, 2, \ldots, r (1)$.
Then define $\ps_{1, \, 0}$ by~(\ref{7701_Ind_psi})
and define $F_{1}$ by~(\ref{7701_Ind_F}).

Now suppose $k \geq 1$ and we have $F_k$ and~$n_k$;
we construct
\[
F_{k + 1},
\,\,\,\,\,\,
n_{k + 1},
\,\,\,\,\,\,
g_{k, 1}, g_{k, 2}, \ldots, g_{k, \, r (n_{k + 1}) / r (n_k)},
\andeqn
\ps_{k + 1, \, k} \, .
\]
Apply Lemma~\ref{L_7630_ConvComb}
with $\ld = \ld_{k, \, k + 1}$,
with $\ep = 2^{- k}$,
and with $F = F_k$,
obtaining $N_0 \in \N$.
Choose $n_{k + 1} > n_k$
and so large that
\[
\frac{r (n_{k + 1})}{r (n_k)}
> \max \left( N_0, \, 2^k \right) \, .
\]
This gives~(\ref{7701_Ind_Ratio}).
Apply the conclusion of Lemma~\ref{L_7630_ConvComb}
with $N = r (n_{k + 1}) / r (n_k)$,
calling the resulting functions
$g_{k, 1}, g_{k, 2}, \ldots, g_{k, \, r (n_{k + 1}) / r (n_k)}$.
Then define $\ps_{k + 1, \, k}$ by~(\ref{7701_Ind_psi}).
This gives~(\ref{7701_Ind_Small}).
Finally, define $F_{k + 1}$ by~(\ref{7701_Ind_F}).
This completes the induction.

For $j, k \in \Nz$ with $j \leq k$, define
$\ps_{k, j} \colon C ([0, 1], \, \R) \to C ([0, 1], \, \R)$
by
\[
\ps_{k, j}
= \ps_{k, \, k - 1} \circ \ps_{k - 1, \, k - 2}
     \circ \cdots \circ \ps_{j + 1, \, j} \, .
\]

An induction argument shows that for $j, k \in \Nz$ with $j \leq k$,
we have
\[
E (\ld_{j, k}) (f_j) \in F_k
\andeqn
\ps_{k, j} (f_j) \in F_k \, .
\]
This condition,
together with Proposition~\ref{P_6917_OrdULim},
allows us to conclude that, as order unit spaces,
we have
\begin{align}\label{Eq_7729_OUS_Iso}
& \varinjlim \big( (C ([0, 1], \, \R))_{k = 0, 1, \ldots},
    \, ( E (\ld_{j, k}))_{0 \leq j \leq k} \big)
\\
& \hspace*{6em} {\mbox{}}
 \cong \varinjlim \big( (C ([0, 1], \, \R))_{k = 0, 1, \ldots},
    \, ( \ps_{k, j})_{0 \leq j \leq k} \big) \, .
  \notag
\end{align}

For $k \in \Nz$ define
\[
\af_{k + 1, \, k} \colon
    C ([0, 1], \, M_{r (n_k)})
   \to  C ([0, 1], \, M_{r (n_{k + 1})})
      = M_{r (n_{k + 1}) / r (n_k)}
       \big( C ([0, 1], \, M_{r (n_k)} ) \big)
\]
by
\[
\af_{k + 1, \, k} (f)
 = \diag \big( f \circ g_{k, 1}, \, f \circ g_{k, 2}, \,
     \ldots, \, f \circ g_{k, \, r (n_{k + 1}) / r (n_k)} \big)
\]
for $f \in C ([0, 1], \, M_{r (n_k)})$.
Let $A$ be the resulting direct limit \ca.

It is easy to check, and is stated as Lemma~3.5 of~\cite{thomsen},
that ${\widehat{\af_{k + 1, k}}} = \ps_{k + 1, k}$.
Letting $V$ and $W$ be the order unit spaces
\[
V = \varinjlim \big( (C ([0, 1], \, \R))_{k = 0, 1, \ldots},
    \, ( E (\ld_{j, k}))_{0 \leq j \leq k} \big)
\]
and
\[
W = \varinjlim \big( (C ([0, 1], \, \R))_{k = 0, 1, \ldots},
    \, ( {\widehat{\af_{k, j} }} )_{0 \leq j \leq k} \big) \, ,
\]
(\ref{Eq_7729_OUS_Iso})
now says $V \cong W$.
Lemma~3.2 of~\cite{thomsen}
and~(\ref{Eq_7822_NN})
imply that
$V \cong \Aff (K)$.
Proposition~\ref{P_6919_CommLim}
implies that
$\Aff (T (A)) \cong W$.
So $\Aff (T (A)) \cong \Aff (K)$,
whence $T (A) \cong K$
by Theorem~\ref{T_7701_AffKToK}.
\end{proof}

\begin{Prop}\label{P_6917_CStLim}
Let $(D_n)_{n = 0, 1, 2, \ldots}$ and $(C_n)_{n = 0, 1, 2, \ldots}$
be sequences of unital \ca{s}.
Let
\[
\big( (D_n)_{n = 0, 1, 2, \ldots},
( \ph_{n, m} )_{0 \leq m \leq n} \big)
\andeqn
\big( (D_n)_{n = 0, 1, 2, \ldots},
( \ph_{n, m}' )_{0 \leq m \leq n} \big)
\]
and
\[
\big( (C_n)_{n = 0, 1, 2, \ldots},
( \ps_{n, m} )_{0 \leq m \leq n} \big)
\andeqn
\big( (C_n)_{n = 0, 1, 2, \ldots},
( \ps_{n, m}' )_{0 \leq m \leq n} \big)
\]
be direct systems
with unital \hm{s},
and call the direct limits (in order) $D$, $D'$, $C$, and~$C'$.
Suppose further that we have unital \hm{s}
$\mu_n, \mu_n' \colon D_n \to C_n$
for $n \in \Nz$
such that
\[
\mu_n \circ \ph_{n, m} = \ps_{n, m} \circ \mu_n
\andeqn
\mu_n' \circ \ph_{n, m}' = \ps_{n, m}' \circ \mu_n'
\]
for all $m, n \in \Nz$ with $m \leq n$.
Let $\mu \colon D \to C$ and $\mu' \colon D' \to C'$
be the induced maps of the direct limits.
Assume that for all $m \in \Nz$ we have
\[
\sum_{n = m}^{\infty}
\big\| {\widehat{\ph_{n, m}}} - {\widehat{\ph_{n, m}'}} \big\|
< \infty
\andeqn
\sum_{n = m}^{\infty}
\big\| {\widehat{\ps_{n, m}}} - {\widehat{\ps_{n, m}'}} \big\|
< \infty \, .
\]
Then there exist isomorphisms
\[
\rh \colon \Aff (\T (D)) \to \Aff (\T (D'))
\andeqn
\sm \colon \Aff (\T (C)) \to \Aff (\T (C'))
\]
such that ${\widehat{\mu'}} \circ \rh = \sm \circ {\widehat{\mu}}$.
Moreover, if $C_n = D_n$ for all $n \in \Nz$
and $\ps_{n, m} = \ph_{n, m}$ and $\ps_{n, m}' = \ph_{n, m}$
for all $m$ and~$n$,
then we can take $\sm = \rh$.
\end{Prop}

\begin{proof}
We can apply
Proposition~\ref{P_6917_OrdULim}
and Proposition~\ref{P_6918_Nat}
using arbitrary countable dense subsets
of the closed unit balls
of $\Aff (\T (D_n))$ and $\Aff (\T (C_n))$
for $n \in \N$.
Under the hypotheses of the last statement,
the uniqueness statement in Proposition~\ref{P_6917_OrdULim}
implies that $\sm = \rh$.
\end{proof}

\begin{Lemma}\label{L_6919_IfAgreeTodn}
Adopt the notation of Construction~\ref{Cn_6918_General},
including (\ref{Cn_6918_General_2nd})
(a second set of maps),
and (\ref{Cn_6918_General_Diag}) and~(\ref{Cn_6918_General_Agree})
(diagonal maps, agreeing
in the coordinates $1, 2, \ldots, d (n + 1)$).
Then
\[
\Big\| {\widehat{\Gm_{n + 1, \, n}^{(0)}}}
- {\widehat{\Gm_{n + 1, \, n}}} \Big\|
\leq \frac{2 k (n + 1)}{d (n + 1) + k (n + 1)}
\]
for all $n \in \Nz$.
\end{Lemma}

\begin{proof}
For a compact metrizable space~$Z$,
let $M (Z)$ be the real Banach space
consisting of all signed Borel measures on~$Z$.
(That is, $M (Z)$ is the dual space of $C (Z, \R)$.)
Identify $Z$ with the set of point masses in $M (Z)$.
For $n \in \Nz$,
we can identify $\T (C_n)$
with the weak* compact convex subset of $M (X_n \amalg Y_n)$
consisting of probability measures.
Thus $X_n \amalg Y_n \S \T (C_n)$.
For every
function $f \in \Aff ( \T (C_n) )$,
the function $\io_n (f) (z) = f (z) \cdot 1_{M_{r (n)}}$
for $z \in X_n \amalg Y_n$
is in $C (X_n \amalg Y_n, \, M_{r (n)}) = C_n$,
and $\ta (\io_n (f)) = f (\ta)$
for all $\ta \in X_n \amalg Y_n \S \T (C_n)$,
hence also all $\ta \in \T (C_n)$
by linearity and continuity.

For $f \in \Aff ( \T (C_n) )$
and $\ta \in \T (C_{n + 1})$,
we can apply the formula in
Construction \ref{Cn_6918_General}(\ref{Cn_6918_General_Diag})
to $\io_n (f)$ and apply $\ta$ to everything,
to get
\[
{\widehat{\Gm_{n + 1, \, n}^{(0)}}} (f) (\ta)
= \frac{1}{l (n + 1)}
      \sum_{k = 1}^{l (n + 1)} \ta (f \circ S_{n, 1}^{(0)})
\]
and
\[
{\widehat{\Gm_{n + 1, \, n}}} (f) (\ta)
= \frac{1}{l (n + 1)}
       \sum_{k = 1}^{l (n + 1)} \ta (f \circ S_{n, 1}) \, .
\]
Using~(\ref{Cn_6918_General_Agree}),
we get
\begin{align*}
\left| {\widehat{\Gm_{n + 1, \, n}^{(0)}}} (f) (\ta)
    - {\widehat{\Gm_{n + 1, \, n}}} (f) (\ta) \right|
& = \frac{1}{l (n + 1)} \left| \sum_{k = d (n + 1) + 1}^{l (n + 1)}
    \big[ \ta (f \circ S_{n, 1}^{(0)})
      - \ta (f \circ S_{n, 1}) \big] \right|
\\
& \leq \frac{l (n + 1) - d (n + 1)}{l (n + 1)}
    \big( 2 \| f \|_{\infty} ) \, .
\end{align*}
The conclusion follows.
\end{proof}

We add additional parts to Construction~\ref{Cn_6918_General}
and Construction~\ref{Cn_6918_General_Part_1a}.

\begin{Construction}\label{Cn_6921_GeneralPart2}
Adopt the assumptions and notation
of all parts of Construction~\ref{Cn_6918_General}
(except (\ref{Cn_6918_General_Agree})),
and in addition make the following assumptions and definitions:
\begin{enumerate}
\setcounter{enumi}{\value{CnsEnumi}}
\item\label{Cn_6918_General_Cross}
For all $m \in \Nz$,
the maps
$S_{m, j}^{(0)}, S_{m, j} \colon
   X_{m + 1} \amalg Y_{m + 1} \to X_m \amalg Y_m$
satisfy
\[
S_{m, j}^{(0)} (X_{m + 1}) \S X_m
\andeqn
S_{m, j}^{(0)} (Y_{m + 1}) \S Y_m
\]
for $j = 1, 2, \ldots, l (m)$,
\[
S_{m, j} (X_{m + 1}) \S X_m
\andeqn
S_{m, j} (Y_{m + 1}) \S Y_m
\]
for $j = 1, 2, \ldots, d (m)$,
and
\[
S_{m, j} (X_{m + 1}) \S Y_m
\andeqn
S_{m, j} (Y_{m + 1}) \S X_m
\]
for $j = d (m) + 1, \, d (m) + 2, \, \ldots, \, l (m)$.
\item\label{Cn_6918_General_AF}
For $m \in \Nz$,
define $D_m = M_{r (m)} \oplus M_{r (m)}$.
Define
$\ph^{(0)}_{m + 1, m}, \, \ph_{m + 1, m} \colon D_m \to D_{m + 1}$
by, for $a, b \in M_{r (m)}$,
\[
\ph_{m + 1, m}^{(0)} (a, b)
= \big( \diag (a, a, \ldots, a), \, \diag (b, b, \ldots, b) \big)
\]
and
\[
\ph_{m + 1, m} (a, b)
= \big( \diag (a, a, \ldots, a, b, b, \ldots, b),
\, \diag (b, b, \ldots, b, a, a, \ldots, a) \big) \, ,
\]
in which $a$ occurs $d (m)$ times in the first entry on the right
and $k (m)$ times in the second entry,
while $b$ occurs $k (m)$ times in the first entry
and $d (m)$ times in the second entry.
For $m, n \in \Nz$ with $m \leq n$, define
\[
\ph_{n, m}
= \ph_{n, n - 1} \circ \ph_{n - 1, \, n - 2} \circ \cdots
\circ \ph_{m + 1, m}
\colon D_m \to D_n \, ,
\]
and define $\ph^{(0)}_{n, m} \colon D_m \to D_n$ similarly.
Define the following AF algebras:
\[
D = \dirlim_m (D_m,\ph_{m + 1, m}  )
\andeqn
D^{(0)} = \dirlim_m (D_m, \ph^{(0)}_{m + 1, m}  ) \, ,
\]
and for $m \in \N$
let $\ph_{\infty, m} \colon D_m \to D$
and $\ph^{(0)}_{\infty, m} \colon D_m \to D^{(0)}$
be the maps associated to these direct limits.
\item\label{Cn_6918_General_Map}
For $m \in \Nz$,
define $\mu_m \colon D_m \to C_m$
as follows.
For $a, b \in M_{r (m)}$
let $f \in C (X_m, M_{r (m)})$
and $g \in C (Y_m, M_{r (m)})$
be the constant functions with values $a$ and~$b$.
Then set $\mu_m (a, b) = (f, g)$.
Further,
following
Lemma \ref{L_6921_Conseq}(\ref{L_6921_Conseq_Comm}) below,
let $\mu \colon D \to C$ and $\mu^{(0)} \colon D^{(0)} \to C^{(0)}$
be the direct limits of the maps~$\mu_m$.
\item\label{Cn_6918_General_Flip}
For $m \in \Nz$,
define $\te_m \colon D_m \to D_m$
by $\te_m (a, b) = (b, a)$ for $a, b \in M_{r (m)}$.
Further,
following
Lemma \ref{L_6921_Conseq}(\ref{L_6921_Conseq_Flip}) below,
let $\te \in \Aut (D)$
and $\te^{(0)} \in \Aut \bigl( D^{(0)} \bigr)$
be the direct limits of the maps~$\te_m$.
\end{enumerate}
\end{Construction}

\begin{Lemma}\label{L_6921_Conseq}
Under the assumptions of Construction~\ref{Cn_6918_General}
(except (\ref{Cn_6918_General_Agree})),
Construction~\ref{Cn_6918_General_Part_1a},
and Construction~\ref{Cn_6921_GeneralPart2},
the following hold:
\begin{enumerate}
\item\label{L_6921_Conseq_Separate}
The direct system
$\big( (C_n^{(0)})_{n = 0, 1, 2, \ldots},
( \Gm^{(0)}_{n, m} )_{0 \leq m \leq n} \big)$
is the direct sum of two direct systems
\[
\big( (C (X_n, \, M_{r (n)}))_{n = 0, 1, 2, \ldots},
( \Gm^{(0)}_{n, m} |_{C (X_m, \, M_{r (m)})} )_{0 \leq m \leq n} \big)
\]
and
\[
\big( (C (Y_n, \, M_{r (n)}))_{n = 0, 1, 2, \ldots},
( \Gm^{(0)}_{n, m}
 |_{C (Y_m, \, M_{r (m)})} )_{0 \leq m \leq n} \big) \, ,
\]
and $C^{(0)}$ is isomorphic to the direct sum of the direct limits
$A$ and $B$ of these systems.
\item\label{L_6921_Conseq_Comm}
For all $m, n \in \Nz$ with $m \leq n$,
\[
\Gm_{n, m}^{(0)} \circ \mu_m = \mu_n \circ \ph_{n, m}^{(0)}
\andeqn
\Gm_{n, m} \circ \mu_m = \mu_n \circ \ph_{n, m} \, .
\]
Moreover,
the maps $\mu_m$ induce unital \hm{s}
$\mu^{(0)} \colon D^{(0)} \to C^{(0)}$
and $\mu \colon D \to C$,
and for all $m \in \Nz$,
\[
\Gm_{\infty, m}^{(0)} \circ \mu_m
= \mu^{(0)} \circ \ph_{\infty, m}^{(0)}
\andeqn
\Gm_{\infty, m} \circ \mu_m = \mu \circ \ph_{\infty, m} \, .
\]
\item\label{L_6921_Conseq_Flip}
For all $m, n \in \Nz$ with $m \leq n$,
\[
\ph_{n, m}^{(0)} \circ \te_m = \te_n \circ \ph_{n, m}^{(0)}
\andeqn
\ph_{n, m} \circ \te_m = \te_n \circ \ph_{n, m} \, .
\]
The maps $\te_m$ induce automorphisms
$\te \colon D \to D$ and $\te^{(0)} \colon D^{(0)} \to D^{(0)}$
such that
\[
\ph_{\infty, m} \circ \te_m = \te \circ \ph_{\infty, m}
\andeqn
\ph_{\infty, m}^{(0)} \circ \te_m
 = \te^{(0)} \circ \ph_{\infty, m}^{(0)}
\]
for all $m \in \Nz$.
\item\label{L_6921_Conseq_Kth}
For all $m \in \Nz$,
$(\mu_m)_* \colon K_* (D_m) \to K_* (C_m)$
is an isomorphism,
and
\[
\mu_* \colon K_* (D) \to K_* (C)
\andeqn
\bigl( \mu^{(0)} \bigr)_* \colon
 K_* \big( D^{(0)} \big) \to K_* \big( C^{(0)} \big)
\]
are isomorphisms.
\end{enumerate}
\end{Lemma}

\begin{proof}
The fact that all the maps
in~(\ref{L_6921_Conseq_Kth})
are isomorphisms on K-theory
comes from the assumption that
the spaces $X_m$ and $Y_m$ are contractible
((\ref{Cn_6918_General_1a_Xn}) and (\ref{Cn_6918_General_1a_YnIs01})
in Construction~\ref{Cn_6918_General_Part_1a}).
Everything else is essentially immediate from the constructions.
\end{proof}

\section{The main theorem}\label{Sec_Main}

We now have the ingredients to deduce the main theorem
of this paper, Theorem \ref{T_6922_FlipOfEll}.

To state the theorem, we first need to define
automorphisms of Elliott invariants,
so we need a category in which they lie.
For convenience, we restrict to unital \ca{s},
and we give a very basic list of conditions.

\begin{Def}\label{D_6921_AbstractEllInv}
An {\emph{abstract unital Elliott invariant}}
is a tuple $G = \big( G_0, (G_0)_+, g, G_1, K, \rh \big)$
in which $(G_0, (G_0)_+, g)$ is a preordered abelian group
with distinguished positive element~$g$,
$G_1$ is an abelian group,
$K$ is a compact convex set (possibly empty),
and $\rh \colon G_0 \to \Aff (K)$
is an order preserving group \hm{}
such that $\rh (g)$ is the constant function~$1$.
(If $K = \varnothing$,
we take $\Aff (K) = \{ 0 \}$,
and we take $\rh$ to be the constant function with value~$0$.)

If
\[
G^{(0)}
= \big( G_0^{(0)}, \big( G^{(0)}_0 \big)_+, g^{(0)},
G_1^{(0)}, K^{(0)}, \rh^{(0)} \big)
\]
and
\[
G^{(1)}
= \big( G_0^{(1)}, \big( G^{(1)}_0 \big)_+, g^{(1)},
G_1^{(1)}, K^{(1)}, \rh^{(1)} \big)
\]
are abstract unital Elliott invariants,
then a {\emph{morphism}} from $G^{(0)}$ to $G^{(1)}$
is a triple $F = (F_0, F_1, S)$
in which $F_0 \colon G_0^{(0)} \to G_0^{(1)}$
is a group \hm{} satisfying
\[
F_0 \big( \bigl( G_0^{(0)} \bigr)_{+} \big)
 \S \bigl( G_0^{(1)} \bigr)_{+}
\andeqn
F_0 \big( g^{(0)} \big) = g^{(1)} \, ,
\]
$F_1 \colon G_1^{(0)} \to G_1^{(0)}$ is a group \hm,
and $S \colon K^{(1)} \to K^{(0)}$
is a continuous affine map satisfying
\begin{equation}\label{Eq_6922_Compat}
\rh^{(1)} (F_0 (\et)) = \rh^{(0)} (\et) \circ S
\end{equation}
for all $\et \in G_0^{(0)}$.

If
\[
F^{(0)} \colon G^{(0)} \to G^{(1)}
\andeqn
F^{(1)} = \big( F_0^{(1)}, F_1^{(1)}, S^{(1)} \big)
   \colon G^{(1)} \to G^{(2)}
\]
are morphisms of abstract unital Elliott invariants,
then define
\[
F^{(1)} \circ F^{(0)}
  = \big( F_0^{(1)} \circ F_0^{(0)}, \,
      F_1^{(1)} \circ F_1^{(0)}, \, S^{(0)} \circ S^{(1)} \big) \, .
\]
(Note: $S^{(0)} \circ S^{(1)}$, not $S^{(1)} \circ S^{(0)}$.)

The {\emph{Elliott invariant of a unital \ca~$A$}}
is
\[
\Ell (A) = \big( K_0 (A), \, K_0 (A)_{+}, \, [1],
\, K_1 (A), \, \T (A), \, \rh_A \big) \, ,
\]
in which $\rh_A \colon K_0 (A) \to \Aff (\T (A))$
is given by $\rh_A (\et) (\ta) = \ta_* (\et)$
for $\et \in K_0 (A)$ and $\ta \in \T (A)$.

If $A$ and $B$ are unital \ca{s}
and $\ph \colon A \to B$ is a unital \hm,
then we define $\ph_* \colon \Ell (A) \to \Ell (B)$
to consist of the maps $\ph_*$ from $K_0 (A)$ to $K_0 (B)$
and from $K_1 (A)$ to $K_1 (B)$,
together with
the map $\T (\ph)$ of Definition~\ref{N_6919_Traces}.
We write it as $(\ph_{*, 0}, \, \ph_{*, 1}, \, \T (\ph))$.
\end{Def}

Definition~\ref{D_6921_AbstractEllInv}
is enough to make the abstract unital Elliott invariants
into a category such that $\Ell (\cdot)$
is a functor from unital \ca{s}
and unital \hm{s}
to abstract unital Elliott invariants.

\begin{Thm}\label{T_6922_FlipOfEll}
There exists a simple unital separable AH algebra $C$
with stable rank~$1$ and with the
following property.
There exists	
an automorphism $F$ of $\Ell (C)$
such $F \circ F$ is the identity morphism of $\Ell (C)$,
but there is no automorphism $\alpha$ of $C$ such that
$\alpha_* = F$.
\end{Thm}

We outline the proof.
We make a first pass through
Construction~\ref{Cn_6918_General}
and Construction~\ref{Cn_6918_General_Part_1a},
without the spaces~$Y_n$,
and without specifying the point evaluation maps.
This is Construction~\ref{Cn_9104_Half} below.
We get a direct system;
call its direct limit~${\widetilde{C}}$.
Apply Proposition~\ref{P_arbitrary_trace_space}
using the sequence of matrix sizes in this system and
$K = \T \bigl( {\widetilde{C}} \bigr)$.
Doing so requires passing to a subsequence of the sequence of
matrix sizes.
Replace the original system with the corresponding subsystem;
Lemma \ref{L_9104_Subsystem} below justifies this.
Then make a second pass through
Construction~\ref{Cn_6918_General}
and Construction~\ref{Cn_6918_General_Part_1a},
taking the spaces $X_n$ and the maps between them from this subsystem
and the spaces $Y_n$ and the maps between them from the
system gotten from Proposition~\ref{P_arbitrary_trace_space},
as needed substituting appropriate point evaluations for the diagonal
entries of the formulas for the maps.
This requires sufficiently few changes that,
by our work in Section~\ref{Sec_Traces},
the tracial state space remains the same.
Therefore the algebra obtained from these
constructions has an order two automorphism of its tracial state space
which corresponds to exchanging the two rows
in the diagram~(\ref{Eq_9Y01_Diagram}).
The constructions have been designed so that there is also
a corresponding automorphism of the K-theory.
Our work in Section~\ref{Sec_Bounds}
rules out the possibility of
a corresponding automorphism of the algebra,
because such an automorphism would necessarily
send a particular corner of the algebra to another one
with a different radius of comparison.

We start with the following construction,
which is ``half'' of Construction~\ref{Cn_6918_General},
and gives just the top row of the
diagram~(\ref{Eq_9Y01_PrelimDiagram}).

\begin{Construction}\label{Cn_9104_Half}
We will consider direct systems and
their associated direct limits constructed as follows.
\begin{enumerate}
\item\label{Cn_9104_Half_dn}
The sequences $(d (n) )_{n = 0, 1, 2, \ldots}$
and $(k (n) )_{n = 0, 1, 2, \ldots}$ in~$\Nz$
are as in
Construction \ref{Cn_6918_General}(\ref{Cn_6918_General_dn})
and satisfy the condition of
Construction \ref{Cn_6918_General}(\ref{Cn_6918_General_dnkn}).
We further define $(l (n) )_{n = 0, 1, 2, \ldots}$
$(r (n) )_{n = 0, 1, 2, \ldots}$,
and $(r (n) )_{n = 0, 1, 2, \ldots}$
as in Construction \ref{Cn_6918_General}(\ref{Cn_6918_General_dn}).
\item\label{Cn_9104_Half_kappa}
Following
Construction \ref{Cn_6918_General}(\ref{Cn_6918_General_Size})
and Construction \ref{Cn_6918_General}(\ref{Cn_6918_General_rrp}),
we define
\[
\kp = \inf_{n \in \N} \frac{s (n)}{r (n)} \, ,
\qquad
\om = \frac{k (1)}{k (1) + d (1)} \, ,
\andeqn
\om' = \sum_{n = 2}^{\infty} \frac{k (n)}{k (n) + d (n)} \, .
\]
(These will not be used directly in connection with
this direct system.)
\item\label{Cn_9104_Half_Spaces}
As in
Construction \ref{Cn_6918_General_Part_1a}(\ref{Cn_6918_General_1a_Xn}),
we define compact metric spaces
by $X_n = \cone \bigl( (S^2)^{s (n)} \bigr)$
for $n \in \Nz$,
and we define maps $Q^{(n)}_j \colon X_{n + 1} \to X_{n}$
for $n \in \Nz$ and $j = 1, 2, \ldots, d (n + 1)$
to be the cones over the projection maps
\[
(S^2)^{s (n + 1)} = \bigl( (S^2)^{s (n)} \bigr)^{d (n + 1)}
  \to (S^2)^{s (n)}.
\]
\item\label{Cn_9104_Half_Maps}
We are given maps
$\dt_n \colon C (X_n) \to C (X_{n + 1}, \, M_{l (n + 1)} )$
(as in
Construction \ref{Cn_6918_General}(\ref{Cn_6918_General_Maps}),
but with only one summand)
which are diagonal,
that is,
there are continuous maps
\[
T_{n, 1}, \, T_{n, 2}, \, \ldots, \,
      T_{n, l (n + 1)}
  \colon X_{n + 1} \to X_n
\]
such that
\[
\dt_n (f)
 = \diag \bigl( f \circ T_{n, 1},
     \, f \circ T_{n, 2}, \, \ldots, \,
      f \circ T_{n, l (n + 1)} \bigr)
\]
for $f \in C (X_n)$.
(Compare with
Construction \ref{Cn_6918_General}(\ref{Cn_6918_General_Diag}).)
Moreover, $T_{n, j} = Q^{(n)}_j$
for $j = 1, 2, \ldots, d (n + 1)$.
The maps $T_{n, j}$ are unspecified
for $j = d (n + 1) + 1, \, d (n + 1) + 2, \, \ldots, \, l (n + 1)$.
\item\label{Cn_9104_Half_System}
Set
$A_n = M_{r (n)} \otimes C ( X_n )$
(like in
Construction \ref{Cn_6918_General}(\ref{Cn_6918_General_Algs})
but with only one summand).
Following
Construction \ref{Cn_6918_General}(\ref{Cn_6918_General_Maps}),
set
\[
\Dt_{n + 1, \, n} = \id_{M_{r (n)}} \otimes \dt_n
  \colon A_n \to A_{n + 1} \, ,
\]
and for $m, n \in \Nz$ with $m \leq n$,
we take
\[
\Delta_{n, m}
= \Delta_{n, n - 1} \circ \Delta_{n - 1, \, n - 2} \circ \cdots
\circ \Delta_{m + 1, m}
\colon A_m \to A_n \, .
\]
\item\label{Cn_9104_Half_Lim}
Define $A = \dirlim_n A_n$,
taken with respect to the maps $\Delta_{n, m}$.
For $n \in \Nz$,
let $\Dt_{\I, n} \colon A_n \to A$
be the map associated with the direct limit.
\end{enumerate}
\end{Construction}

To avoid confusing notation,
we isolate the following computation as a lemma.

\begin{Lemma}\label{L_0722_Comb}
Let $n \in \N$ and let
$\kp_1, \kp_2, \ldots, \kp_n, \dt_1, \dt_2, \ldots, \dt_n \in (0, \I)$.
Then
\[
\sum_{j = 1}^n \frac{\kp_j}{\dt_j + \kp_j}
 \geq \frac{\prod_{j = 1}^n (\dt_j + \kp_j) - \prod_{j = 1}^n \dt_j}{
           \prod_{j = 1}^n (\dt_j + \kp_j)}.
\]
\end{Lemma}

\begin{proof}
For $j = 1, 2, \ldots, n$ define
\[
\ld_j = \frac{\kp_j}{\dt_j + \kp_j}.
\]
Then $\ld_j \in (0, 1)$.
Some calculation shows that the conclusion of the lemma becomes
\begin{equation}\label{Eq_0722_Lambda}
\sum_{j = 1}^n \ld_j
 \geq 1 - \prod_{j = 1}^n (1 - \ld_j).
\end{equation}

We prove~(\ref{Eq_0722_Lambda}) by induction on~$n$.
For $n = 1$ it is trivial.
Suppose~(\ref{Eq_0722_Lambda}) is known for some value of~$n$.
Given $\ld_1, \ld_2, \ldots, \ld_{n + 1} \in (0, 1)$,
set $\mu = 1 - (1 - \ld_n) (1 - \ld_{n + 1} )$.
Then
\[
\mu \in (0, 1)
\andeqn
\mu = \ld_n + \ld_{n + 1} - \ld_n \ld_{n + 1}
 \leq \ld_n + \ld_{n + 1}.
\]
Applying the induction hypothesis on
$\ld_1, \ld_2, \ldots, \ld_{n - 1}, \mu$
at the second step,
we then have
\[
\sum_{j = 1}^{n + 1} \ld_j
 \geq \sum_{j = 1}^{n - 1} \ld_j + \mu
 \geq 1 - \Biggl[ \prod_{j = 1}^{n - 1} (1 - \ld_j) \Biggr] (1 - \mu)
 = 1 - \prod_{j = 1}^{n + 1} (1 - \ld_j).
\]
This completes the induction, and the proof of the lemma.
\end{proof}

\begin{Lemma}\label{L_9104_Subsystem}
Let a direct system as in Construction~\ref{Cn_9104_Half} be given,
but using sequences
$\bigl( {\widetilde{d}} (n) \bigr)_{n = 0, 1, 2, \ldots}$
and $\bigr( {\widetilde{k}} (n) \bigr)_{n = 0, 1, 2, \ldots}$
in place of $(d (n) )_{n = 0, 1, 2, \ldots}$
and $(k (n) )_{n = 0, 1, 2, \ldots}$.
Denote the additional
sequences analogous to those in
Construction \ref{Cn_9104_Half}(\ref{Cn_9104_Half_dn})
by ${\widetilde{l}}$, ${\widetilde{r}}$, and~${\widetilde{s}}$.
Denote the numbers analogous
to those in
Construction \ref{Cn_9104_Half}(\ref{Cn_9104_Half_kappa})
by ${\widetilde{\kp}}$, ${\widetilde{\om}}$, and~${\widetilde{\om}}'$.
Denote the spaces used in the system by ${\widetilde{X}}_{n}$.
Let $\nu \colon \Nz \to \Nz$
be a strictly increasing function
such that $\nu (0) = 0$ and $\nu (1) = 1$.
Then the direct system
$\bigl( C \bigl( {\widetilde{X}}_{\nu (m)},
 \, M_{{\widetilde{r}} (\nu (m))} \bigr) \bigr)_{m = 0, 1, 2, \ldots}$
is isomorphic to a system as in
Construction \ref{Cn_9104_Half}, with the choices
$d (0) = 1$, $k (0) = 0$,
\begin{equation}\label{Eq_0722_dm}
d (m) = {\widetilde{d}} \bigl( \nu (m - 1) + 1 \bigr)
           {\widetilde{d}} \bigl( \nu (m - 1) + 2 \bigr)
           \cdots {\widetilde{d}} \bigl( \nu (m) \bigr)
\end{equation}
and
\begin{equation}\label{Eq_0722_km_Form}
k (m) = {\widetilde{l}} \bigl( \nu (m - 1) + 1 \bigr)
           {\widetilde{l}} \bigl( \nu (m - 1) + 2 \bigr)
           \cdots {\widetilde{l}} \bigl( \nu (m) \bigr)
        - d (m)
\end{equation}
for $m \in \N$.
Moreover,
following the notation of Construction \ref{Cn_9104_Half},
\begin{equation}\label{Eq_0723_lm_is_prod}
l (m) = {\widetilde{l}} \bigl( \nu (m - 1) + 1 \bigr)
           {\widetilde{l}} \bigl( \nu (m - 1) + 2 \bigr)
           \cdots {\widetilde{l}} \bigl( \nu (m) \bigr),
\end{equation}
\[
r (m) = {\widetilde{r}} \bigl( \nu (m) \bigr),
\andeqn
s (m) = {\widetilde{s}} \bigl( \nu (m) \bigr)
\]
for $m \in \Nz$,
and
\[
\kp = {\widetilde{\kp}},
\qquad
\om = {\widetilde{\om}},
\andeqn
\om' \leq {\widetilde{\om}}'.
\]
\end{Lemma}

\begin{proof}
Given the definitions of $d$ and~$k$,
the proofs of the formulas for $l$, $r$, and~$s$
are easy.

Using Lemma~\ref{L_6922_Noninc} at the first and fourth steps,
we now get
\[
{\widetilde{\kp}}
    = \lim_{n \to \infty}
         \frac{{\widetilde{s}} (n)}{{\widetilde{r}} (n)}
    = \lim_{m \to \infty}
         \frac{{\widetilde{r}} \bigl( \nu (m) \bigr)}{
                       {\widetilde{s}} \bigl( \nu (m) \bigr)}
    = \lim_{m \to \infty} \frac{s (m)}{r (m)}
    = \kp.
\]
We have $\om = {\widetilde{\om}}$ because $\nu (1) = 1$.

Using Lemma~\ref{L_0722_Comb} at the second step
and (\ref{Eq_0722_dm}), (\ref{Eq_0722_km_Form})
and~(\ref{Eq_0723_lm_is_prod})
at the third step, we have
\begin{align*}
{\widetilde{\om}}'
& = \sum_{m = 2}^{\infty} \sum_{j = \nu (m - 1) + 1}^{\nu (m)}
      \frac{{\widetilde{k}} (j)}{{\widetilde{k}} (j)
                                   + {\widetilde{d}} (j)}
\\
& \geq \sum_{m = 2}^{\infty}
        \frac{\prod_{j = \nu (m - 1) + 1}^{\nu (m)}
                   [{\widetilde{d}} (j) + {\widetilde{k}} (j)]
               - \prod_{j = \nu (m - 1) + 1}^{\nu (m)}
                   {\widetilde{d}} (j)}{
           \prod_{j = \nu (m - 1) + 1}^{\nu (m)}
              [{\widetilde{d}} (j) + {\widetilde{k}} (j)]}
\\
& = \sum_{m = 2}^{\infty} \frac{k (m)}{k (m) + d (m)}
  = \om'.
\end{align*}

Define $X_m = {\widetilde{X}}_{\nu (m)}$ for $m \in \Nz$.
Clearly $X_m = \cone \bigl( (S^2)^{s (m)} \bigr)$,
as required.
Denote the maps in the system of the hypotheses by
\[
{\widetilde{\dt}}_n
 \colon C \bigl( {\widetilde{X}}_{n} \bigr)
  \to C ({\widetilde{X}}_{n + 1}, \, M_{{\widetilde{l}} (n + 1)} \bigr)
\andeqn
{\widetilde{\Dt}}_{n, m}
 \colon {\widetilde{C}}_{m} \to {\widetilde{C}}_{n},
\]
with ${\widetilde{\dt}}_n$ being built using maps
\[
{\widetilde{T}}_{n, 1}, \, {\widetilde{T}}_{n, 2}, \, \ldots, \,
      {\widetilde{T}}_{n, \, l (n + 1)}
  \colon {\widetilde{X}}_{n + 1} \to {\widetilde{X}}_{n},
\]
as in Construction \ref{Cn_9104_Half}(\ref{Cn_9104_Half_Maps}).
For $p = \nu (m), \, \nu (m) + 1, \, \ldots, \, \nu (m + 1) - 1$,
set
\[
j (p)
 = \frac{{\widetilde{r}} (p)}{{\widetilde{r}} (\nu (m))}
 = {\widetilde{l}} \bigl( \nu (m) + 1 \bigr)
           {\widetilde{l}} \bigl( \nu (m) + 2 \bigr)
           \cdots {\widetilde{l}} (p).
\]
Then define
\[
\dt^{(0)}_m
 \colon C \bigl( {\widetilde{X}}_{\nu (m)} \bigr)
  \to C ({\widetilde{X}}_{\nu (m + 1)}, \, M_{l (n + 1)} \bigr)
\]
by
\[
\dt^{(0)}_m
 = \id_{M_{j (\nu (m + 1) - 1)}}
            \otimes {\widetilde{\dt}}_{\nu (m + 1) - 1}
    \circ \id_{M_{j (\nu (m + 1) - 2)}}
            \otimes {\widetilde{\dt}}_{\nu (m + 1) - 2}
    \circ \cdots
    \circ {\widetilde{\dt}}_{\nu (m)}.
\]
(In the last term we omit $\id_{M_{j (\nu (m))}}$
since $j (\nu (m)) = 1$.)
With this definition,
one checks that
$\id_{M_{{\widetilde{r}} (\nu (m))}} \otimes {\widetilde{\dt}}_m
 = {\widetilde{\Dt}}_{ \nu (m + 1), \, \nu (m) }$,
so that the direct system gotten using the maps $\dt^{(0)}_m$
in Construction~\ref{Cn_9104_Half}
is a subsystem of the system given in the hypotheses.

We claim that $\dt^{(0)}_m$ is unitarily equivalent
to a map
$\dt_m \colon C ( X_{m} ) \to C (X_{m + 1}, \, M_{l (n + 1)} \bigr)$
as in Construction~\ref{Cn_9104_Half}.
This will imply isomorphism of the direct systems,
and complete the proof of the lemma.
First,
$\dt^{(0)}_m$ is given as in
Construction \ref{Cn_9104_Half}(\ref{Cn_9104_Half_Maps})
using some maps
from ${\widetilde{X}}_{\nu (m + 1)}$ to ${\widetilde{X}}_{\nu (m)}$,
namely all possible compositions
\[
{\widetilde{T}}_{\nu (m), \, i_{\nu (m)}}
 \circ {\widetilde{T}}_{\nu (m) + 1, \, i_{\nu (m) + 1}} \circ \cdots
 \circ {\widetilde{T}}_{\nu (m + 1) - 1, \, i_{\nu (m + 1) - 1}}
\]
with $i_p = 1, 2, \ldots, {\widetilde{l}} (p + 1)$
for $p = \nu (m), \, \nu (m) + 1, \, \ldots, \nu (m + 1) - 1$.
Moreover,
since the composition of projection maps is a projection map,
restricting to $i_p = 1, 2, \ldots, {\widetilde{d}} (p + 1)$
for all~$p$
gives exactly all the maps
$Q^{(m)}_j \colon X_{m + 1} \to X_m$
for $j = 1, 2, \ldots, d (n + 1)$.
Therefore $\dt^{(0)}_m$
is unitarily equivalent to a map
as in Construction~\ref{Cn_9104_Half} by a permutation matrix.
\end{proof}

\begin{proof}[Proof of Theorem~\ref{T_6922_FlipOfEll}]
Choose $N \in \N$ such that
\begin{equation}\label{Eq_0722_ChooseN}
N > 5
\andeqn
\exp \left( - \frac{1}{N - 1} \right) > \frac{3}{4}.
\end{equation}
(For example, $N = 6$ will work.)
We make preliminary choices of the numbers $d (n)$ etc.\  in
Construction \ref{Cn_6918_General}(\ref{Cn_6918_General_dn}),
calling them ${\widetilde{d}} (n)$ etc.
Take ${\widetilde{d}} (0) = 1$ and ${\widetilde{k}} (0) = 0$,
and take ${\widetilde{d}} (n) = N^n$
and ${\widetilde{k}} (n) = 1$ for $n \in \N$.
Then
\[
{\widetilde{l}} (n) = N^n + 1,
\qquad
{\widetilde{r}} (n) = \prod_{j = 1}^n (N^j + 1),
\andeqn
{\widetilde{s}} (n) = \prod_{j = 1}^n N^j
\]
for $n \in \N$.
We obtain numbers
as in Construction \ref{Cn_9104_Half}(\ref{Cn_9104_Half_kappa})
(equivalently,
Construction \ref{Cn_6918_General}(\ref{Cn_6918_General_Size})
and Construction \ref{Cn_6918_General}(\ref{Cn_6918_General_rrp})),
which we call ${\widetilde{\kp}}$, ${\widetilde{\om}}$,
and ${\widetilde{\om}}'$.
Further,
adopt the definitions and notation of
Construction \ref{Cn_9104_Half},
except that we use ${\widetilde{X}}_n$ instead of~$X_n$
and similarly throughout.
That is,
in Construction \ref{Cn_9104_Half}(\ref{Cn_9104_Half_Spaces})
we call the spaces ${\widetilde{X}}_n$ instead of~$X_n$,
the projection maps ${\widetilde{Q}}_j^{(n)}$,
in Construction \ref{Cn_9104_Half}(\ref{Cn_9104_Half_Maps})
we call
the maps of algebras ${\widetilde{\dt}}_n$
and the maps of spaces
${\widetilde{T}}_{n, j} \colon
 {\widetilde{X}}_{n + 1} \to {\widetilde{X}}_{n}$,
in Construction \ref{Cn_9104_Half}(\ref{Cn_9104_Half_System})
we call the algebras ${\widetilde{A}}_n$
and the maps ${\widetilde{\Dt}}_{n, m}$,
and
in Construction \ref{Cn_9104_Half}(\ref{Cn_9104_Half_Lim})
we call the direct limit~${\widetilde{A}}$
and the maps to it ${\widetilde{\Dt}}_{\I, n}$.
As in Construction \ref{Cn_9104_Half}(\ref{Cn_9104_Half_Maps}),
we take ${\widetilde{T}}_{n, j} = {\widetilde{Q}}^{(n)}_j$
for $j = 1, 2, \ldots, {\widetilde{d}} (n + 1)$.
For $n \in \Nz$ choose an arbitrary point
${\widetilde{x}}_{n} \in {\widetilde{X}}_n$,
and for
$j = {\widetilde{d}} (n + 1) + 1$
let ${\widetilde{T}}_{n, j}$ be
the constant function on ${\widetilde{X}}_{n + 1}$
with value ${\widetilde{x}}_{n}$.
(Note that ${\widetilde{d}} (n + 1) + 1 = {\widetilde{l}} (n + 1)$.)

We claim that
the conditions in
Construction \ref{Cn_6918_General}(\ref{Cn_6918_General_Size}),
Construction \ref{Cn_6918_General}(\ref{Cn_6918_General_rrp}),
and Construction \ref{Cn_6918_General}(\ref{Cn_6918_General_kpom})
are satisfied,
and moreover that
\[
\frac{1}{1 - 2 {\widetilde{\om}}}
 < \frac{2 {\widetilde{\kappa}} - 1}{2 {\widetilde{\om}}} \, .
\]
For $n \in \N$ we have,
using $\log (m + 1) - \log (m) < \frac{1}{m}$ at the third step,
\begin{align*}
\frac{{\widetilde{s}} (n)}{{\widetilde{r}} (n)}
& = \prod_{j = 1}^n \frac{N^j}{N^j + 1}
  = \exp \Biggl( \sum_{j = 1}^n
    - \big[ \log (N^j + 1) - \log (N^j) \big] \Biggr)
\\
& \geq \exp \Biggl( - \sum_{j = 1}^n \frac{1}{N^j} \Biggr)
  > \exp \left( - \frac{1}{N - 1} \right) \, .
\end{align*}
So
${\widetilde{\kp}}
 \geq \exp \left( - \frac{1}{N - 1} \right)
 > \frac{3}{4}$
by~(\ref{Eq_0722_ChooseN}).
Moreover,
\[
{\widetilde{\om}} = \frac{1}{N + 1} < \frac{1}{4}
\andeqn
{\widetilde{\om}}' = \sum_{j = 2}^{\infty} \frac{1}{N^j + 1}
       < \sum_{j = 2}^{\infty} \frac{1}{N^j}
       = \frac{1}{N (N - 1)} \, ,
\]
so the conditions ${\widetilde{\om}}' < {\widetilde{\om}} < \frac{1}{2}$
in Construction \ref{Cn_6918_General}(\ref{Cn_6918_General_rrp})
and $2 {\widetilde{\kappa}} - 1 > 2 {\widetilde{\om}}$
in Construction \ref{Cn_6918_General}(\ref{Cn_6918_General_kpom})
are satisfied.
Moreover,
\[
\frac{1}{1 - 2 {\widetilde{\om}}}
 = \frac{N + 1}{N - 1}
 < \frac{N + 1}{4}
 = \frac{1}{4 {\widetilde{\om}}}
 = \frac{2 \left( \frac{3}{4} \right) - 1}{2 {\widetilde{\om}}}
 < \frac{2 {\widetilde{\kappa}} - 1}{2 {\widetilde{\om}}} \, .
\]
The claim is proved.

Apply Proposition~\ref{P_arbitrary_trace_space}
with $K = \T \bigl( {\widetilde{A}} \bigr)$
and with ${\widetilde{l}} (n)$ and ${\widetilde{r}} (n)$
in place of $l (n)$ and $r (n)$,
getting a strictly increasing sequence,
which we call $(\nu (n))_{n = 0, 1, 2, \ldots}$,
with $\nu (j) = j$ for $j = 0, 1$,
an AI~algebra $B_0$
(called $A$ in Proposition~\ref{P_arbitrary_trace_space})
which is the direct limit of a unital system
\[
C ([0, 1]) \otimes M_{r (\nu (0))}
 \overset{\alpha_{1, 0}}{\longrightarrow}
         C ([0, 1]) \otimes M_{r (\nu (1))}
 \overset{\alpha_{2, 1}}{\longrightarrow}
         C ([0, 1]) \otimes M_{r (\nu (2))}
 \overset{\alpha_{3, 2}}{\longrightarrow} \cdots
\]
with injective diagonal maps
$\alpha_{n + 1, \, n}$
given by
\[
f \mapsto \diag \big( f \circ R_{n, 1}, \, f \circ R_{n, 2},
\, \ldots, \, f \circ R_{n, \, r (\nu_{n + 1}) / r (\nu_{n})} \big)
\]
for \ct{} functions
\[
R_{n, 1}, R_{n, 2}, \ldots, R_{n, \, r (\nu (n + 1)) / r (\nu (n))}
  \colon [0, 1] \to [0, 1] \, ,
\]
and an isomorphism
$\T (B_0) \to \T \bigl( {\widetilde{A}} \bigr)$.

Apply Lemma~\ref{L_9104_Subsystem}
with this choice of~$\nu$.
Define the sequences $(d (n) )_{n = 0, 1, 2, \ldots}$
and $(k (n) )_{n = 0, 1, 2, \ldots}$
as in Lemma~\ref{L_9104_Subsystem},
and then make all the definitions in 
Construction \ref{Cn_6918_General} and
\ref{Cn_6918_General_Part_1a}.
(Some are also given in the statement of Lemma~\ref{L_9104_Subsystem}.)
Then, as in the proof of Lemma~\ref{L_9104_Subsystem},
$X_n = {\widetilde{X}}_{\nu (n)}$.
We make the following choices for the unspecified
objects in these constructions.
We choose points $x_n \in X_n$ and $y_n \in [0, 1]$
for $n \in \Nz$
such that the conditions in
Construction
\ref{Cn_6918_General_Part_1a}(\ref{condition-points-dense})
and
Construction
\ref{Cn_6918_General_Part_1a}(\ref{Cn_6918_General_1a_yn})
are satisfied.
(It is easy to see that this can be done.)
Use these points in
Construction
\ref{Cn_6918_General_Part_1a}(\ref{Cn_6918_General_1a_Maps_XHigh})
and
Construction
\ref{Cn_6918_General_Part_1a}(\ref{Cn_6918_General_1a_Maps_YTop}).
Take the maps
\[
R_{n, 1}, \, R_{n, 2}, \, \ldots, \, R_{n, \, d (n + 1)}
\colon Y_{n + 1} \to Y_n
\]
in Construction
\ref{Cn_6918_General_Part_1a}(\ref{Cn_6918_General_1a_Maps_YBottom})
to be those from the application
of Proposition \ref{P_arbitrary_trace_space} above.
For $j = 1, \, 2, \, \ldots, \, l (n + 1)$,
let
$S_{n, j}^{(0)} |_{X_{n + 1}} \colon X_{n + 1} \to X_n$ be the
maps in the system obtained from
Lemma~\ref{L_9104_Subsystem},
and take $S_{n, j}^{(0)} |_{Y_{n + 1}} = R_{n, j}$.
The requirement
$S_{n, j}^{(0)} = S_{n, j}$ for $j = 1, 2, \ldots, d (n + 1)$ in
Construction
\ref{Cn_6918_General_Part_1a}(\ref{Cn_9Y01_Gen_1a_Maps2})
is then satisfied,
so that the condition in
Construction
\ref{Cn_6918_General}(\ref{Cn_6918_General_Agree})
is also satisfied.
Moreover, with these choices,
the conditions in
Construction
\ref{Cn_6921_GeneralPart2}(\ref{Cn_6918_General_Cross})
are satisfied.

By Lemma~\ref{L_9104_Subsystem},
the numbers $\kp$, $\om$, and~$\om'$
from Construction \ref{Cn_6918_General}(\ref{Cn_6918_General_Size})
and Construction \ref{Cn_6918_General}(\ref{Cn_6918_General_Size})
satisfy
\[
\kp = {\widetilde{\kp}},
\qquad
\om = {\widetilde{\om}},
\andeqn
\om' \leq {\widetilde{\om}}'.
\]
Therefore $\kp > \frac{1}{2}$,
$\om' < \om < \frac{1}{2}$,
and $2 \kp - 1 > 2 \om$,
as required in
Construction \ref{Cn_6918_General}(\ref{Cn_6918_General_Size}),
Construction \ref{Cn_6918_General}(\ref{Cn_6918_General_rrp}),
and Construction \ref{Cn_6918_General}(\ref{Cn_6918_General_kpom});
moreover
\begin{equation}\label{Eq_0724_omkpRel}
\frac{1}{1 - 2 \om} < \frac{2 \kappa - 1}{2 \om} \, .
\end{equation}

The algebra~$C$ is simple by Lemma~\ref{L_0725_Simplicity}.

The algebras $A$ and $B$ of
Lemma~\ref{L_6921_Conseq}(\ref{L_6921_Conseq_Separate})
are now $A = {\widetilde{A}}$ and $B = B_0$,
so that $C^{(0)}$,
as in Construction \ref{Cn_6918_General}(\ref{Cn_6918_General_2nd}),
is isomorphic to ${\widetilde{A}} \oplus B_0$.
The isomorphism
$\T (B_0) \to \T \bigl( {\widetilde{A}} \bigr)$
gives an isomorphism
$\zt^{(0)}_0 \colon \Aff (\T (A)) \to \Aff (\T (B))$.
This provides an automorphism
of $\Aff (\T (A)) \oplus \Aff (\T (B))$,
given by
\[
(f, g) \mapsto
   \bigl( \bigl(\zt^{(0)}_0 \bigr)^{-1} (g),
      \, \zt^{(0)}_0 (f) \bigr) \, .
\]
Let $\zt^{(0)}$ be the corresponding automorphism of
$\Aff (\T (A \oplus B)) = \Aff \bigl( \T \bigl( C^{(0)} \bigr) \bigr)$
gotten using Lemma~\ref{L_6921_DSumTraces}.
Clearly $\zt^{(0)} \circ \zt^{(0)}$ is the
identity map on $\Aff \bigl( \T \bigl( C^{(0)} \bigr) \bigr)$.

Adopt the notation
of Construction \ref{Cn_6921_GeneralPart2}:
$C$ and $C^{(0)}$ are as already described,
$D$ and $D^{(0)}$ are the AF~algebras from
Construction \ref{Cn_6921_GeneralPart2}(\ref{Cn_6918_General_AF}),
$\mu \colon D \to C$ and $\mu^{(0)} \colon D^{(0)} \to C^{(0)}$
are the maps of
Construction \ref{Cn_6921_GeneralPart2}(\ref{Cn_6918_General_Map})
(which are isomorphisms on K-theory by
Lemma \ref{L_6921_Conseq}(\ref{L_6921_Conseq_Kth})),
and $\te \in \Aut (D)$ and $\te^{(0)} \in \Aut \bigl( D^{(0)} \bigr)$
are as in
Construction \ref{Cn_6921_GeneralPart2}(\ref{Cn_6918_General_Flip}).

Define $E = \dirlim_n M_{r (m)}$,
with respect to the maps
$a \mapsto \diag (a, a, \ldots, a)$,
with $a$ repeated $l (n)$ times.
The direct system defining $D^{(0)}$
is the direct sum of two copies of the direct system just defined,
so
\[
D^{(0)} \cong E \oplus E
\andeqn
\Aff \bigl(\T  \bigl( D^{(0)} \bigr) \bigr)
 \cong \Aff (\T (E \oplus E)).
\]
Since $E$ is a UHF~algebra,
we have
$\Aff (\T (E)) \cong \R$ with the usual order and order unit~$1$.
Using $\id_{\Aff (\T (E))}$ in place of $\zt^{(0)}_0$ above,
we get an automorphism of $\Aff \bigl( \T \bigl( D^{(0)} \bigr) \bigr)$.
But this automorphism is just~${\widehat{\te^{(0)}}}$.

We claim that
$\zt^{(0)} \circ {\widehat{\mu^{(0)}}}
= {\widehat{\mu^{(0)}}} \circ {\widehat{\te^{(0)}}}$.
To prove the claim,
we work with
\[
\Aff (\T (E)) \oplus \Aff (\T (E))
\andeqn
\Aff (\T (A)) \oplus \Aff (\T (B))
\]
in place of $\Aff \bigl( \T \bigl( D^{(0)} \bigr) \bigr)$
and $\Aff \bigl( \T \bigl( C^{(0)} \bigr) \bigr)$,
but keep the same names for the maps.

Since $\mu^{(0)} \colon E \oplus E \to A \oplus B$
is the direct sum of unital maps from the first summand to~$A$
and the second summand to~$B$,
the map ${\widehat{\mu^{(0)}}}$ is similarly a direct sum
of maps $\Aff (\T (E)) \to \Aff (\T (A))$
and $\Aff (\T (E)) \to \Aff (\T (B))$.
Let $e$ and $f$
be the order units of $\Aff (\T (A))$ and $\Aff (\T (B))$.
The unique positive order unit preserving maps
$\Aff (\T (E)) \to \Aff (\T (A))$
and $\Aff (\T (E)) \to \Aff (\T (B))$
are $\af \mapsto \af e$ and $\bt \mapsto \bt f$
for $\af, \bt \in \R$.
Therefore ${\widehat{\mu^{(0)}}} (\af, \bt) = (\af e, \bt f)$.
Since $\zt_0^{(0)}$ is order unit preserving,
we have $\zt_0^{(0)} (e) = f$,
so
\[
\zt^{(0)} (\af e, \bt f)
= (\bt e, \af f)
= {\widehat{\mu^{(0)}}} (\bt, \af)
= \bigl( {\widehat{\mu^{(0)}}} \circ {\widehat{\te^{(0)}}} \bigr)
    (\af, \bt) \, .
\]
The claim follows.

Using conditions (\ref{Cn_6918_General_rrp})
and~(\ref{Cn_6918_General_Agree})
in Construction~\ref{Cn_6918_General},
Lemma~\ref{L_6919_IfAgreeTodn},
and Proposition~\ref{P_6917_CStLim},
we get isomorphisms
\[
\rh \colon
 \Aff \bigl( \T \bigl( D^{(0)} \bigr) \bigr) \to \Aff (\T (D))
\andeqn
\sm \colon
 \Aff \bigl( \T \bigl( C^{(0)} \bigr) \bigr) \to \Aff (\T (C))
\]
such that
${\widehat{\mu}} \circ \rh = \sm \circ {\widehat{\mu^{(0)}}}$.
Define
\[
\et = \rh \circ {\widehat{\te^{(0)}}} \circ \rh^{-1}
  \in \Aut \big( \Aff (\T (D)) \big)
\quad {\mbox{and}} \quad
\zt = \sm \circ \zt^{(0)} \circ \sm^{-1}
  \in \Aut \big( \Aff (\T (C)) \big) \, .
\]
A calculation now shows that the claim above implies
\begin{equation}\label{Eq_6922_PreCDiag}
\zt \circ {\widehat{\mu}} = {\widehat{\mu}} \circ \et.
\end{equation}
We also have $\zt \circ \zt = \id_{\Aff (\T (C))}$.

We want to apply Proposition~\ref{P_6917_CStLim}
with $D_n$ and $\ph_{n, m}$ as in
Construction \ref{Cn_6921_GeneralPart2}(\ref{Cn_6918_General_AF}),
and $\ph_{n, m}^{(0)}$ as there in place of $\ph_{n, m}'$,
so that $D$ and $D^{(0)}$ are as already given,
with $C_n = D_n$ for all $n \in \Nz$ 
and $\ps_{n, m} = \ph_{n, m}$ and $\ps_{n, m}' = \ph_{n, m}$
for all $m$ and~$n$,
and with $\te_n$, $\te_n^{(0)}$, $\te$, and $\te^{(0)}$ 
from
Construction \ref{Cn_6921_GeneralPart2}(\ref{Cn_6918_General_Flip})
in place of $\mu_n$, $\mu_n'$, $\mu$, and $\mu'$.
As before, this application is justified by
conditions (\ref{Cn_6918_General_rrp})
and~(\ref{Cn_6918_General_Agree})
in Construction~\ref{Cn_6918_General},
and Lemma~\ref{L_6919_IfAgreeTodn}.
The outcome is an isomorphism
$\rh' \colon
 \Aff \bigl( \T \bigl( D^{(0)} \bigr) \bigr) \to \Aff (\T (D))$
such that
\begin{equation}\label{Eq_0724_rhPrime}
{\widehat{\te}} = \rh' \circ {\widehat{\te^{(0)}}} \circ (\rh')^{-1}.
\end{equation}

We claim that $\et = {\widehat{\te}}$.
The ``right'' way to do this is presumably to show that
$\rh' = \rh$ above,
but the following argument is easier to write.
We have
\[
\Aff (\T (D))
 \cong \Aff \bigl( \T \bigl( D^{(0)} \bigr) \bigr)
 \cong \R^2,
\]
with order $(\af, \bt) \geq 0$ \ifo{} $\af \geq 0$ and $\bt \geq 0$
and order unit $(1, 1)$.
Since the state space $S (\R^2)$
of $\R^2$ with this order unit space structure
is an interval,
and automorphisms of order unit spaces
preserve the extreme points of the state space,
there is only one possible action of a nontrivial automorphism
of $\R^2$ on $S (\R^2)$.
Theorem~\ref{T_7701_AffKToK} implies that
$\R^2 \cong \Aff (S (\R^2))$,
so there is only one
nontrivial automorphism of $\R^2$.
Since ${\widehat{\te^{(0)}}}$ is nontrivial,
so is ${\widehat{\te}}$ by~(\ref{Eq_0724_rhPrime}),
and so is $\et$ by its definition.
The claim follows.

The claim and~(\ref{Eq_6922_PreCDiag}) imply
\begin{equation}\label{Eq_6922_CDiag}
\zt \circ {\widehat{\mu}} = {\widehat{\mu}} \circ {\widehat{\te}}.
\end{equation}

Passing to state spaces and applying Theorem~\ref{T_7701_AffKToK},
we get an affine homeomorphism
$H \colon \T (C) \to \T (C)$
such that $\zt (f) = f \circ H$ for all $f \in \Aff (\T (C))$,
and moreover $H \circ H = \id_{\T (C)}$.
By Lemma~\ref{L_6921_Conseq}(\ref{L_6921_Conseq_Kth}),
the expression $\mu_* \circ \te_* \circ (\mu_*)^{-1}$
is a well defined automorphism of $K_* (C)$, of order~$2$.
We claim that
$F = \big( \mu_* \circ \te_* \circ (\mu_*)^{-1}, \, H)$
is an order~$2$ automorphism of $\Ell (C)$.
We use the notation of Definition~\ref{D_6921_AbstractEllInv}
for the Elliott invariant of a \ca;
in particular, $\rh_C$ and $\rh_D$ are not related to the maps $\rh$
and $\rh'$ above.
The only part needing work is the
compatibility condition~(\ref{Eq_6922_Compat})
in Definition~\ref{D_6921_AbstractEllInv},
which amounts to showing that
\[
\rh_{C} \circ \mu_* \circ \te_* \circ (\mu_*)^{-1}
= \zt \circ \rh_{C} \, .
\]
To see this,
we calculate,
using
at the second and last steps
the notation of Definition~\ref{N_6919_Traces}
and the fact that the morphisms of Elliott invariants
defined by $\mu$ and $\te$ satisfy~(\ref{Eq_6922_Compat})
in Definition~\ref{D_6921_AbstractEllInv},
and using (\ref{Eq_6922_CDiag}) at the third step,
\begin{align*}
\zt \circ \rh_{C}
& = \zt \circ \rh_{C} \circ \mu_* \circ (\mu_*)^{-1}
= \zt \circ {\widehat{\mu}} \circ \rh_{D} \circ (\mu_*)^{-1}
\\
& = {\widehat{\mu}} \circ {\widehat{\te}}
\circ \rh_{D} \circ (\mu_*)^{-1}
= \rh_{C} \circ \mu_* \circ \te_* \circ (\mu_*)^{-1} \, ,
\end{align*}
as desired.

Thus, we have constructed an automorphism $F$ of $\Ell (C)$
of order~$2$.
It remains to show that $F$ is not induced by any automorphism of~$C$.

Using~(\ref{Eq_6922_CDiag}) on the last components,
one easily sees that $F \circ \mu_* = \mu_* \circ \te_*$.
Let $q$ and $q^{\perp}$ be as in Notation \ref{notation-q-q-perp}.
In the construction of $D$
as in Construction \ref{Cn_6921_GeneralPart2}(\ref{Cn_6918_General_AF}),
set $e = \varphi_{\infty,1}((1,0))$
and $e^{\perp} = 1 - e = \varphi_{\infty,1}((0,1))$.
Then $\te (e) = e^{\perp}$, $\mu (e) = q$,
and $\mu(e^{\perp}) = q^{\perp}$.
Therefore $F ([q]) = [q^{\perp}]$.

Suppose now that there exists an automorphism $\alpha$ such
that $\alpha_* = F$. Then $[\alpha (q)] = [q^{\perp}]$.
By Lemma \ref{lemma-stable-rank-1},
$\alpha (q)$ is unitarily equivalent to $q^{\perp}$.
Let $u$ be a unitary such that $u\alpha (q)u^* = q^{\perp}$.
Thus, since
$\alpha(q A q) = \alpha (q)A\alpha (q) = u^* q^{\perp} A q^{\perp} u$,
it follows that the $q A q$ and $q^{\perp} A q^{\perp}$ have the
same radius of comparison.
By~(\ref{Eq_0724_omkpRel}),
this contradicts
Lemmas \ref{Lemma:rc_lower_bound} and~\ref{Lemma:rc-upper-bound}.
\end{proof}

\begin{rmk}\label{R_0727_DontNeedKThy}
One can easily check that,
with $C$ as in the proof of Theorem~\ref{T_6922_FlipOfEll},
there is a unique automorphism of $\Ell (C)$
whose component automorphism
of the tracial state space is as in the proof.
Therefore the conclusion can be slightly strengthened:
there is an automorphism of $\T (C)$
which is compatible with an automorphism of $\Ell (C)$
but which is not induced by any automorphism of~$C$.
\end{rmk}

\begin{qst}\label{Q_7816_DynEx}
Does there exist a compact metric space~$X$
and a \mh{} $h \colon X \to X$
such that the crossed product $C^* (\Z, X, h)$
has the same features as the example we construct here?
\end{qst}

\begin{qst}[Blackadar]\label{Q_7910_NotOrder2}
Does there exist a simple separable
stably finite unital nuclear \ca~$C$
and an automorphism $F$ of $\Ell (C)$
such that:
\begin{enumerate}
\item\label{Item_Q_7910_NotOrder2_FF}
$F \circ F$ is the identity morphism of $\Ell (C)$.
\item\label{Item_Q_7910_NotOrder2_alpha}
There is an automorphism $\alpha$ of $C$ such that
$\alpha_* = F$.
\item\label{Item_Q_7910_NotOrder2_noalpha2}
There is no $\af$ as in~(\ref{Item_Q_7910_NotOrder2_alpha})
which in addition
satisfies $\af \circ \af = \id_C$.
\end{enumerate}
Can such an algebra be chosen to be AH and have stable rank~$1$?
\end{qst}

Our method of proof suggests that, instead of being just a number,
the radius of comparison should be taken to be
a function from $V (A)$ to $[0, \I]$.
If one uses the generalization to nonunital algebras
in \cite[Section 3.3]{BRTTW},
one could presumably even get a function
from ${\operatorname{Cu}} (A)$ to $[0, \I]$.

\bibliographystyle{alpha}

\end{document}